\documentclass[aihp]{imsart}

\RequirePackage{amsthm,amsfonts,amssymb}
\RequirePackage{amsmath}
\RequirePackage[numbers]{natbib}
\RequirePackage[colorlinks,citecolor=blue,urlcolor=blue]{hyperref}
\RequirePackage{graphicx}
\usepackage{algorithm}
\usepackage{breakcites}
\usepackage{bigstrut} %needed to keep blockarrays from being crunched
\usepackage{blkarray}
\usepackage{booktabs}
\usepackage{enumitem} % detailed list control
\usepackage{float} % for 'H' forced placement of floats
\usepackage{mathrsfs} % for mathscr
\usepackage{multirow}
\usepackage{scalerel}
\usepackage{soul} % for \st = strikethrough
\usepackage{subcaption}
\usepackage{verbatim}
\usepackage[dvipsnames]{xcolor}
\usepackage{xfrac}
\usepackage{xr}

\usepackage{outlines}
\setlist{noitemsep, topsep=0pt} % for enumitem
\definecolor{LinkBlue}{rgb}{.15, .25, .85} %for hyperref

\makeatletter
\renewcommand*{\NAT@spacechar}{~}
\makeatother

\makeatletter
\setbool{@fleqn}{false}
\makeatother

\newfloat{algorithm}{tbp}{loa}
\providecommand{\algorithmname}{Algorithm}
\floatname{algorithm}{\protect\algorithmname}

\makeatletter
\newcommand*{\addFileDependency}[1]{% argument=file name and extension
	\typeout{(#1)}
	\@addtofilelist{#1}
	\IfFileExists{#1}{}{\typeout{No file #1.}}
}
\makeatother

\newcommand*{\myexternaldocument}[1]{%
	\externaldocument{#1}%
	\addFileDependency{#1.tex}%
	\addFileDependency{#1.aux}%
}
\myexternaldocument{text_bernoulli}

\startlocaldefs

\theoremstyle{plain}

\newtheorem{theorem}{Theorem}[section]
\newtheorem{lemma}[theorem]{Lemma}

\newtheorem{corollary}[theorem]{Corollary}

\theoremstyle{remark}
\newtheorem{definition}[theorem]{Definition}
\theoremstyle{definition}

\newcommand{\eq}[1]{\begin{align*}#1\end{align*}} % alignable equation*

\newcommand{\ec}[1]{\begin{gather*}#1\end{gather*}} % centered equation*

\makeatletter
\newcommand{\removeParBefore}{\ifvmode\vspace*{-\baselineskip}\setlength{\parskip}{0ex}\fi}
\newcommand{\removeParAfter}{\@ifnextchar\par\@gobble\relax}

\newcommand{\eqinner}[1]{\endlinechar=32%
	\begin{equation*}\begin{aligned}#1\end{aligned}\end{equation*}\endgroup\removeParAfter}

\newcommand{\ecinner}[1]{\endlinechar=32%
	\begin{gather*} #1 \end{gather*}\endgroup\removeParAfter}

\makeatother

\newcommand{\R}{\ensuremath{\mathbb{R}}}

\newcommand{\E}{\operatorname{\mathbb{E}}}

\renewcommand{\P}{\operatorname{\mathbb{P}}}

\newcommand{\Bern}{\operatorname{Bern}}

\newcommand{\N}{\operatorname{N}}

\newcommand{\Unif}{\operatorname{Unif}}

\newcommand{\scrF}{\mathscr{F}}
\newcommand{\scrG}{\mathscr{G}}

\newcommand{\calC}{\mathcal{C}}

\newcommand{\calX}{\mathcal{X}}
\newcommand{\calZ}{\mathcal{Z}}

\newcommand{\cd}{\cdot}
\newcommand{\eps}{\epsilon}
\newcommand{\g}{\, | \,}
\newcommand{\TV}{\mathrm{TV}}

\newcommand{\diff}{\mathrm{d} } % e.g. \int f(x) \diff x

\newcommand{\qaq}{\quad \text{and} \quad}
\newcommand{\sas}{\ \text{and} \ }
\newcommand{\zo}{\{0,1\}}

\newcommand{\symdiff}{ \operatorname{{\scalebox{.9}{$\bigtriangleup$}}}}
\newcommand{\calXp}{\calX \times \calX}
\newcommand{\scrFp}{\scrF \otimes \scrF}

\newcommand{\bp}{\bar P}
\newcommand{\bq}{\bar Q}

\newcommand{\bb}{\bar B}

\newcommand{\sm}{\setminus}
\newcommand{\sx}{\{ x \}}
\newcommand{\sy}{\{ y \}}

\newcommand{\xy}{(x,y)}
\newcommand{\XY}{(X,Y)}
\newcommand{\xyp}{(x',y')}

\newcommand{\bxy}{(b_x,b_y)}
\newcommand{\dxyp}{(\mathrm{d} x', \mathrm{d} y')}
\newcommand{\axy}{A_x \times A_y}
\newcommand{\Gmax}{\Gamma^\mathrm{max}}

\newcommand{\axm}{(A_x\! \setminus\! \sx)}
\newcommand{\aym}{(A_y\! \setminus\! \sy)}
\newcommand{\axym}{\axm \times \aym}
\newcommand{\axs}{(A_x \cap  \sx)}

\newcommand{\h}{\sfrac{1}{2}}

\newcommand{\sxy}{S_{xy}}
\endlocaldefs

\begin{document}

\begin{frontmatter}
\title{Metropolis--Hastings transition kernel couplings}
\runtitle{Metropolis--Hastings transition kernel couplings}

\begin{aug}
	%%%%%%%%%%%%%%%%%%%%%%%%%%%%%%%%%%%%%%%%%%%%%%%
	%% Only one address is permitted per author. %%
	%% Only division, organization and e-mail is %%
	%% included in the address.                  %%
	%% Additional information can be included in %%
	%% the Acknowledgments section if necessary. %%
	%%%%%%%%%%%%%%%%%%%%%%%%%%%%%%%%%%%%%%%%%%%%%%%
	\author[a]{\fnms{John} \snm{O'Leary}\ead[label=e1]{john.oleary@tudor.com}} and
	\author[b]{\fnms{Guanyang} \snm{Wang}\ead[label=e2]{guanyang.wang@rutgers.edu}}

	%%%%%%%%%%%%%%%%%%%%%%%%%%%%%%%%%%%%%%%%%%%%%%
	%% Addresses                                %%
	%%%%%%%%%%%%%%%%%%%%%%%%%%%%%%%%%%%%%%%%%%%%%%
	\address[a]{Tudor Investment Corporation, \printead{e1}}
	\address[b]{Department of Statistics, Rutgers University, \printead{e2}}
\end{aug}

\begin{abstract}
Couplings play a central role in the analysis of Markov chain convergence and in the construction of
novel Markov chain Monte Carlo estimators, diagnostics, and variance reduction techniques. The set
of possible couplings is often intractable, frustrating the search for tight bounds and efficient
estimators. To address this challenge for algorithms in the Metropolis--Hastings (MH) family, we
establish a simple characterization of the set of MH transition kernel couplings. We then extend
this result to describe the set of maximal couplings of the MH kernel, resolving an open question of
\citet{OLeary2020}. Our results represent an advance in understanding the MH transition kernel and a
step forward for coupling this popular class of algorithms.
\end{abstract}



\end{frontmatter}

\section{Introduction}
\label{sec:intro}

Couplings have played an important role in the analysis of Markov chain convergence since the early
days of the field~\citep{doeblin1938expose, harris1955chains}. Beyond their role as a proof
technique, couplings have also been used as a basis for sampling~\citep{propp:wilson:1996,
	fill1997interruptible, neal1999circularly,flegal2012exact}, convergence diagnosis
\citep{johnson1996studying, johnson1998coupling, biswas2019estimating}, variance reduction
\citep{neal2001improving,Goodman2009,piponi2020hamiltonian}, and unbiased estimation
\citep{glynn2014exact, Jacob2020, heng2019,pmlr-v89-middleton19a, Middleton2020, heng2021aunbiased,
	heng2021bunbiased}. Couplings that deliver smaller meeting times are associated with better results,
and the design of such couplings is a central problem for theoretical and computational
statisticians who use these methods.

Given a pair of Markov chains, the distribution of meeting times that can occur under any coupling
is constrained by the total variation distance between the marginal distributions of the chains at
each iteration~\citep{aldous1983random, lindvall2002lectures}. Couplings that achieve the fastest
meeting times allowed by this bound are said to be efficient or maximal and exist in some generality
\citep{griffeath1975maximal, pitman1976coupling, goldstein1979maximal}. These couplings are rarely
Markovian, whereas this is a practical requirement in many of the applications above. Efficient
Markovian couplings are known in a few cases, but they remain elusive for discrete-time MCMC
algorithms~\citep{burdzy2000efficient, connor2008optimal, kuwada2009characterization,
	hsu2013maximal, kendall2015coupling, bottcher2017markovian, banerjee2017rigidity}. For techniques
such as the Metropolis--Hastings (MH) algorithm~\citep{Metropolis1953, hastings:1970}, it is
difficult even to describe or parameterize the set of possible couplings, much less to identify an
optimum within this set. Despite its practical importance, coupling design is often an unsystematic
task that relies on qualitative insight, experience, and luck.

In this paper, we aim to address this challenge by deriving a simple characterization of the set of
possible MH kernel couplings. We show that every such coupling can be represented as a coupled
proposal step followed by a coupled acceptance step with a short list of properties. We also show
the converse: that any coupled proposal and acceptance steps with these properties combine to yield
a coupling of MH transition kernels. Despite its simplicity, this result requires a clear
understanding of the mechanics of MH coupling and has proved elusive since at least
\citet{johnson1998coupling}. Yet it is of some theoretical and practical value, as described below.

On the theory side, our result addresses an important open question on the structure of MH kernel
couplings, which is whether all such couplings have a `two-step' representation. So far, two
approaches to constructing MH kernel couplings have appeared in the literature: in the first
(one-step) method, a coupling is defined directly in terms of the probability density or mass
function of the marginal transition kernel~\citep{rosenthal1996analysis, rosenthal2002quantitative,
	qin2018wasserstein, OLeary2020}. In the other (two-step) method, one first defines a coupling at the
MH proposal step followed by a second coupling at the acceptance step~\citep{johnson1998coupling,
	bou2018coupling, Jacob2020}. One-step couplings seem more general, but they can be difficult to
analyze. In contrast, two-step couplings have proven amenable to methods like those of
\citet{bou2018coupling}, which yield explicit bounds on the contraction rate between chains. Our
results show that all MH kernel couplings have a two-step representation, resolving the apparent
tension between these two approaches.

On the practical side, our results make it easier to identify high-performance couplings of the MH
transition kernel. First, they reframe the problem of transition kernel coupling design into the
simpler tasks of selecting proposal and acceptance step couplings. Our results show that any
coupling can be expressed in a two-step form, so there is no need to consider complex one-step
couplings such as those in~\citet{o2021couplings} in order to achieve rapid convergence between chains. This
is an additional benefit since two-step couplings are often easier to implement than one-step
couplings. Finally, proposal and acceptance step couplings often have natural parameterizations, and
in this case, our results support the use of numerical methods such as gradient descent to optimize
the performance of MH kernel couplings.

With the above results in hand, we turn to maximal couplings of the MH transition kernel, which are
sometimes called greedy couplings~\citep{aldous1995reversible, hayes2005general}, one-step maximal
couplings~\citep{reutter1995general, hayes2003non, kartashov2013maximal}, or step-by-step maximal
couplings~\citep{pillai2019mixing} of the associated Markov chains. While these couplings do not
typically achieve the fastest meeting times allowed by the coupling inequality, they often perform
well and serve as a valuable reference point for further analysis. Building on the results described
above, we characterize maximal couplings in terms of the properties of the associated proposal and
acceptance couplings. We also resolve an open question of~\citet{OLeary2020} on the structure of
maximal transition kernel couplings and their relationship with maximal couplings of the proposal
distributions.

We note that our results hold for simple MH algorithms as well as refinements such as Hamiltonian
Monte Carlo~\citep{Duane:1987, Neal1993, neal2011mcmc}, the Metropolis-adjusted Langevin algorithm
\citep{roberts1996exponential}, and particle MCMC~\citep{andrieu:doucet:holenstein:2010}. They also
hold on both continuous and discrete state spaces, and for `lazy' implementations where the proposal
distribution is not absolutely continuous with respect to the base measure. Finally, they hold for
methods like Barker's algorithm~\citep{Barker1965} where the acceptance rate function differs from
the usual MH form. Thus our characterizations of kernel couplings and maximal kernel couplings apply
to most Markov chains which involve a sequence of proposal and acceptance steps.

The rest of this paper is organized as follows. In Section~\ref{sec:kercoup}, we describe our
notation and setting, state our main result on transition kernel couplings, and set out a series of
definitions and lemmas to prove it. In Section~\ref{sec:maximal}, we highlight a few important
properties of maximal couplings, consider the relationship between maximal proposal and transition
kernel couplings, and prove our main result on the structure of maximal couplings of the MH
transition kernel. In Section~\ref{sec:appn}, we analyze the kernel couplings of several widely-used
MH algorithms such as the random-walk MH algorithm and the Metropolis-adjusted Langevin algorithm.
In Section~\ref{sec:discussion} we discuss our results and outline future directions. Finally, in
the appendix, we illustrate our definitions and results with a series of simple examples.

%%%%%%

\section{Metropolis--Hastings kernel couplings}
\label{sec:kercoup}

In this section, we show that every MH-like transition kernel coupling can be represented as a
proposal coupling followed by an acceptance indicator coupling (Theorem \ref{thm:repr}). After
establishing our setting and notation, we build up to this proof through several auxiliary results
(Lemmas~\ref{lem:generation}-\ref{lem:rndproperties}). The key step is Lemma \ref{lem:camdpc}, which
shows how to map from an arbitrary MH-like transition kernel coupling to a proposal coupling
together with a structure we call a `coupled acceptance mechanism.' The representation of transition
kernel couplings in terms of coupled proposal and acceptance steps follows from this result.
Finally, Corollary \ref{cor:repr_non_atomic} in Section \ref{sec:cor} gives a simplified form of
Theorem \ref{thm:repr} under mild assumptions on the proposal kernel.

\subsection{Notation and setting}
\label{sec:notation}

Let $(\calZ, \scrG)$ be a measurable space. We say $\Theta: \calZ \times \scrG \to [0, 1]$ is a
Markov kernel if $\Theta(z, \cdot)$ is a probability measure on $(\calZ, \scrG)$ for each $z \in
\calZ$ and if $\Theta(\cdot, A) : \calZ \to [0,1]$ is a measurable function for each $A \in \scrG$.
We say that $\Theta$ is a sub-probability kernel if it satisfies the same conditions but with
$\Theta(z,\cdot)$ a sub-probability measure for each $z \in \calZ$ rather than a probability
measure. Finally, let $2^S$ denote the power set of a set $S$, let $\Bern(\alpha)$ denote the
Bernoulli distribution on $\{0,1\}$ with $\P(\Bern(\alpha)=1) = \alpha$ for $\alpha \in [0,1]$, and
let $a \wedge b := \min(a,b)$ for any $a,b \in \R$.

Suppose that $(\calX, \scrF)$ is a Polish space with base measure $\lambda$. We take $\calX$ to be
our state space. Fix a Markov kernel ${Q: \calX \times \scrF \to [0,1]}$, which we call the proposal
kernel, and a function $a: \calX \times \calX \to [0,1]$, which we call the acceptance rate
function. To be an acceptance rate function, we require that $a(x,\cd)$ is $Q(x,\cd)$-measurable and
that $a(x,x)=1$ for all $x \in \calX$. The second condition simplifies many proofs and involves no
loss of generality, as we discuss below. Given a current state $x \in \calX$ and a measurable set $A
\in \scrF$, we think of $Q(x, A)$ as the probability of proposing a move from $x$ to some $x' \in
A$, and we think of $a(x,x')$ as the probability of accepting such a proposal. We call the
probability measure ${ B(x,x') := \Bern(a(x,x')) }$ the acceptance indicator distribution associated
with the acceptance rate function $a$, given a proposed move from $x$ to $x'$.

In this paper, we consider transition kernels that generalize the MH algorithm according to the
following definition:

\smallskip

\begin{definition}(MH-like transition kernel)\label{def:mh-like}
We say that a Markov kernel $P$ is \textit{generated by $Q$ and $a$} if for all $x \in \calX$, $x'
\sim Q(x,\cdot)$ and $b_x \sim B(x,x') = \Bern(a(x,x'))$ imply $X := b_x x' + (1 - b_x) x \sim
P(x,\cdot)$. We say $P$ is \textit{generated by $Q$ and $B$} if $P$ is generated by $Q$ and $a$, and
if $B$ is the acceptance indicator distribution associated with $a$. Finally, we say that a Markov
transition kernel $P$ is \textit{MH-like} if there exists a Markov kernel $Q: \calX \times \scrF \to
[0,1]$ and an acceptance rate function $a: \calX \times \calX \to [0,1]$ such that $P$ is generated
by $Q$ and $a$.
\end{definition}

Most MCMC algorithms that involve alternating proposal and acceptance steps are MH-like. The MH case
arises when we fix a target distribution $\pi \ll \lambda$, assume $Q(x, \cdot) \ll \lambda$ for all
$x \in \calX$, and set ${ a(x,x') := 1 \wedge ( \pi(x') q(x',x) / (\pi(x) q(x,x') ) ) }$, where
$\pi( \cd ) := \diff \pi / \diff \lambda$ and $q(x, \cd) := \diff Q(x,\cd) / \diff \lambda$. Our
analysis also holds for alternative forms of $a(x,x')$ and for proposal kernels that are not
absolutely continuous with respect to $\lambda$. The requirement that $a(x,x) = 1$ for all $x \in
\calX$ involves no loss of generality since any transition kernel $P$ that can be expressed in terms
of some $\tilde a$ without this property can also be expressed in terms of another acceptance
function $a$ that has it. Thus our analysis applies to continuous and discrete state spaces, to lazy
algorithms in which $Q(x, \sx) > 0$, and to the transition kernel of Barker's algorithm
\citep{Barker1965} where ${a(x,x') := \pi(x') q(x',x) / (\pi(x) q(x,x') + \pi(x') q(x',x))}$ for
each $x\neq x'$. We refer the reader to~\citet{andrieu2020general} for a general framework for
thinking about MH-like transition kernels.

Next, let $\mu$ and $\nu$ be any finite measures on $(\calX, \scrF)$. A measure $\gamma$ on
$(\calXp, \scrFp)$ is called a coupling of $\mu$ and $\nu$ if $\gamma(A \times \calX) = \mu(A)$ and
$\gamma(\calX \times A) = \nu(A)$ for all $A \in \scrF$. We write $\Gamma(\mu, \nu)$ for the set of
couplings of $\mu$ and $\nu$, sometimes called the Fr\'echet class of these measures
\citep{kendall2017lectures}. When $\mu$ and $\nu$ are probability measures, the coupling inequality
states that ${\P(X = Y) \leq 1 - \lVert \mu - \nu \rVert_\TV}$ for $(X,Y) \sim \gamma \in
\Gamma(\mu, \nu)$. Here $\lVert \mu - \nu \rVert_\TV = \sup_{A \in \scrF} | \mu(A) - \nu(A) |$ is
the total variation distance. See e.g. \citet[chap. 1.2]{lindvall2002lectures} or
\citet[chap.~4]{Levin2017} for a further discussion of this important inequality. A coupling
$\gamma$ that achieves the coupling inequality bound is said to be maximal, and we write $\Gmax(\mu,
\nu)$ for the set of such couplings. We note that there are important similarities between maximal
couplings and optimal transport couplings, whose goal is to minimize the expected distance between
draws according to some metric. The need for exact meeting depends on the context. For most of the
applications described in Section~\ref{sec:intro}, the advantage of maximal couplings is that the
pairs of chains can be kept together indefinitely once meeting occurs.

Suppose that $\Theta$ is a Markov kernel on $(\calZ,\scrG)$. Following
\citet[][chap.19]{douc2018markov}, we call $\bar \Theta : (\calZ \times \calZ) \times (\scrG \otimes
\scrG) \to [0,1]$ a kernel coupling based on $\Theta$ if $\bar \Theta$ is a Markov kernel on $(\calZ
\times \calZ, \scrG \otimes \scrG)$ and if $\bar \Theta((z_1,z_2),\cd) \in
\Gamma(\Theta(z_1,\cd),\Theta(z_2,\cd))$ for all $z_1,z_2 \in \calZ$. We write
$\Gamma(\Theta,\Theta)$ for the set of all such kernel couplings. Likewise, we say that $\bar
\Theta$ is a maximal kernel coupling if each $\bar \Theta((z_1,z_2), \cd)$ is a maximal coupling of
$\Theta(z_1,\cd)$ and $\Theta(z_2,\cd)$, and we write $\Gmax(\Theta, \Theta)$ for the set of these.
It is important to remember that a maximal kernel coupling $\bp \in \Gmax(P,P)$ does not usually
correspond to a maximal coupling of Markov chains in the sense of~\citet{aldous1983random}. The
former produces the highest probability of meeting at each individual step, while the latter
produces the smallest meeting times allowed by the coupling inequality.

Suppose that an MH-like kernel $P$ is generated by a proposal kernel $Q$ and an acceptance rate
function $a$. Then we call any $\bp \in \Gamma(P,P)$ a transition kernel coupling and any $\bq \in
\Gamma(Q,Q)$ a proposal kernel coupling. Let $\mathcal{B}$ be the power set of $\{0,1\}^2$. We call
$\bb: \calX^2 \times \calX^2 \times \mathcal{B} \to [0,1]$ an acceptance indicator coupling if
$\bb(\xy, \xyp, \cdot)$ is a measure on $\{0,1\}^2$ for all $x,y,x',y' \in \calX$ and if $\bb(\cdot,
\cdot, S): \calX^2 \times \calX^2 \to [0,1]$ is measurable for all $S \in \mathcal{B}$. Abusing
notation, we drop the third argument of $\bb$ and write $(b_x, b_y) \sim \bb(\xy, \xyp)$ whenever a
pair of random variables $(b_x, b_y) \in \{0,1\}^2$ follows the acceptance indicator coupling $\bb$.
Thus we say that a proposal coupling $\bq$ and an acceptance indicator coupling $\bb$ generate a
kernel coupling $\bp$ if for all $x,y \in \calX$, ${ \xyp \sim \bq(\xy,\cdot) }$ and ${ (b_x, b_y)
	\sim \bb(\xy, \xyp) }$ imply that $ (X,Y) := (b_x x' + (1 - b_x) x, b_y y' + (1 - b_y) y) \sim
\bp(\xy, \cdot). $ The first challenge before us is to show that every coupling $\bp \in
\Gamma(P,P)$ arises in this way. The second is to determine conditions on $\bq$ and $\bb$ such that
pairs $(X,Y)$ defined as above follow a kernel coupling $\bp \in \Gamma(P,P)$.

\subsection{Main result}

The literature on MH transition kernel couplings has so far considered individual examples of
couplings or at most small classes of them. \citet[section 3.4]{johnson1998coupling} showed that one
can start with a proposal coupling $\bq \in \Gamma(Q,Q)$ and construct an acceptance indicator
coupling $\bb$ such that $\bq$ and $\bb$ generate an MH transition kernel coupling $\bp \in
\Gamma(P,P)$. $\bq$ is taken to be the $\gamma$\nobreakdash-coupling of~\citet{lindvall2002lectures}
and $\bb$ is obtained by setting $b_x = 1(U \leq a(x,x'))$ and $b_y = 1(U \leq a(y,y'))$ using the
same ${U \sim \Unif}$. \citet{Jacob2020} maintains this $\bb$ while modifying $\bq$ to obtain faster
convergence between chains, employing a reflection strategy that yields dependence between $x'$ and
$y'$ even when $x' \neq y'$. Most recently,~\citet{OLeary2020} defined a conditional $\bb$ to obtain
a maximal transition kernel coupling $\bp$ from any maximal proposal kernel coupling $\bq$. Such
couplings $\bb$ accept proposals with $x' = y'$ at more than the MH rate in exchange for a lower
acceptance rate when $x' \neq y'$, such that each chain continues to marginally evolve according to
$P$.

In contrast with these previous studies, our aim in this paper is to fully characterize the set of
MH transition kernel couplings $\Gamma(P,P)$ in terms of proposal and acceptance indicator
couplings. In particular, we prove the following:

\smallskip

\begin{theorem}
\label{thm:repr}
Let $P$ be the MH-like transition kernel on $(\calX, \scrF)$ generated by a
proposal kernel $Q$ and an acceptance rate function $a$. A joint kernel $\bp \in \Gamma(P,P)$ if
and only if it is generated by $\bq \in \Gamma(Q,Q)$ and an acceptance indicator coupling $\bb$
such that for any $x, y \in \calX$ and $\bxy \sim \bb(\xy,\xyp)$, we have
\begin{enumerate}
	\item $\P(b_x=1 \g x,y,x') = a(x,x')$ for $Q(x,\cdot)$-almost all $x'$, and
	\item $\P(b_y=1 \g x,y,y')= a(y,y')$ for $Q(y,\cdot)$-almost all $y'$.
\end{enumerate}
\end{theorem}

This result shows that all MH-like transition kernel couplings $\bp$ are `natural' in the sense that
they arise from coupled proposals $\xyp \sim \bq(\xy, \cdot)$ that are accepted or rejected
according to coupled indicators ${ (b_x, b_y) \sim \bb(\xy, \xyp) }$. The $\bb$ conditions say that
from any current state pair $\xy$, the $x'$ and $y'$ proposals are marginally accepted at the rate
given by $a$, potentially with complicated joint behavior. Finally, Theorem \ref{thm:repr} confirms
that simple conditions on the joint proposal and acceptance steps suffice to yield a kernel coupling
$\bp$ in $\Gamma(P,P)$.

It is simple to show that the combination of a proposal coupling $\bq$ and an acceptance indicator
coupling $\bb$ with the properties above yields a kernel coupling $\bp \in \Gamma(P,P)$. The
converse requires a deeper understanding of the structure of transition kernel couplings. An example
of the difficulty is as follows. A transition from the state pair $\xy$ to a measurable rectangle $A_x
\times A_y$ with $x \not \in A_x$ and $y \not \in A_y$ can only occur if points $x' \in A_x$ and $y'
\in A_y$ are proposed and accepted. This makes it simple to work out the behavior of a hypothetical
$\bq$ and $\bb$ on these sets. However, transitions to ${A_x \times \sy}, {\sx \times A_y}, \sas \sx
\times \sy$ can arise from the partial or full rejection of proposed moves to measurable subsets of
$A_x \times \calX$, $\calX \times A_y$, and $\calX \times \calX$. It is challenging to relate all of
these transition probabilities to joint proposal and acceptance measures in a consistent way.

Nevertheless, we prove Theorem~\ref{thm:repr} by constructing a mapping from an arbitrary kernel
$\bp \in \Gamma(P,P)$ to a joint distribution $\bq$ on proposals $(x', y')$ and a coupling $\bb$ of
acceptance indicators $\bxy$, such these reproduce $\bp$ and agree with the original proposal kernel
$Q$ and acceptance rate function $a$. To do so, we define a collection of joint measures $\Phi$ on
$\xyp$ and $\bxy$ with certain decomposition properties. This $\Phi$ implies a proposal coupling
$\bq$, and a Radon--Nikodym argument then yields an acceptance indicator coupling $\bb$ with the desired
properties.

\subsection{Coupled acceptance mechanisms}

To begin, we define a relationship that can exist between an MH-like transition kernel coupling and
a proposal coupling:

\smallskip

\begin{definition}
\label{def:cam}
We say that a proposal coupling $\bq \in \Gamma(Q,Q)$ and a transition kernel coupling $\bp \in \Gamma(P,P)$ are
related by a \textit{coupled acceptance mechanism} $\Phi = (\Phi_{11}, \Phi_{10},\Phi_{01}, \Phi_{00})$
if each $\Phi_{ij}(\xy, \cdot)$ is a sub-probability on $(\calXp, \scrFp)$,
and if for all $x,y \in \calX$, $A_x, A_y \in \scrF$ we have
\eq{
	1. \ &\bq(\xy, \axy) =
	(\Phi_{11}+\Phi_{10}+\Phi_{01}+\Phi_{00}) ( \xy, \axy ) \hspace{114pt} \\
	2. \ & \bp(\xy, \axy) = \Phi_{11}(\xy, \axy) + \Phi_{10}(\xy, A_x \times \calX) 1(y \in A_y)  \\
	& \hspace{40pt} + \Phi_{01}(\xy, \calX \times A_y) 1(x \in A_x)
	+ \Phi_{00}(\xy, \calX \times \calX) 1(x \in A_x) 1(y \in A_y) \\
	3. \ & Q(x,\sx)\! =\! (\Phi_{11}\!+\!\Phi_{10})(\xy, \sx\! \times\! \calX) \sas Q(y, \sy)\! = \! (\Phi_{11}\!+\!\Phi_{01})(\xy, \calX\! \times\! \sy).
}
\end{definition}

We will see that the existence of a coupled acceptance mechanism $\Phi$ relating $\bq$ and $\bp$ is
equivalent to the existence of an acceptance indicator coupling $\bb$ such that $\bq$ and $\bb$
generate $\bp$. In Lemma~\ref{lem:generation} we show that if $(X,Y) \sim \bp(\xy, \cdot)$ is
generated by $\xyp \sim \bq(\xy, \cdot)$ and $(b_x, b_y) \sim \bb(\xy, \xyp)$, then a coupled
acceptance mechanism $\Phi$ relating $\bq$ and $\bp$ exists and can be defined by $\Phi_{ij}(\xy,
\cd) := \P(\xyp \in \cd, b_x = i, b_y = j)$ for ${i, j \in \{0,1\}}$. Conversely, in
Lemma~\ref{lem:rnderivs}, we show that if a coupled acceptance mechanism $\Phi$ relates $\bq$ and
$\bp$, then there exists an acceptance indicator coupling $\bb$ such that $\bq$ and $\bb$ generate
$\bp$.

Thus Definition~\ref{def:cam} captures a set of relationships between a proposal coupling $\bq$ and
a transition kernel coupling $\bp$ required for the former to potentially generate the latter. We
see that $\Phi$ divides the probability in $\bq(\xy,\cd)$ according to scenarios in which one, both,
or neither proposal is accepted, as illustrated in Figure~\ref{fig:cam}. Condition 1 requires
the total probability over these scenarios to sum to $\bq$ for all measurable subsets of $\calXp$.
Condition 2 says that the resulting distribution over transitions must also agree with $\bp$. Finally,
Condition 3 addresses proposals with $x' = x$ or $y' = y$ and requires $\Phi$ to agree with
the assumption that such proposals are always accepted. In Appendix~\ref{ex:running}, we show that
Condition 3 is independent of Conditions 1 and 2. We also show that more than one coupled acceptance mechanism $\Phi$
can relate the same pair $\bq \in \Gamma(Q,Q)$ and $\bp \in \Gamma(P,P)$.

\begin{figure}[t]
\centering
\includegraphics[width=.65 \linewidth]{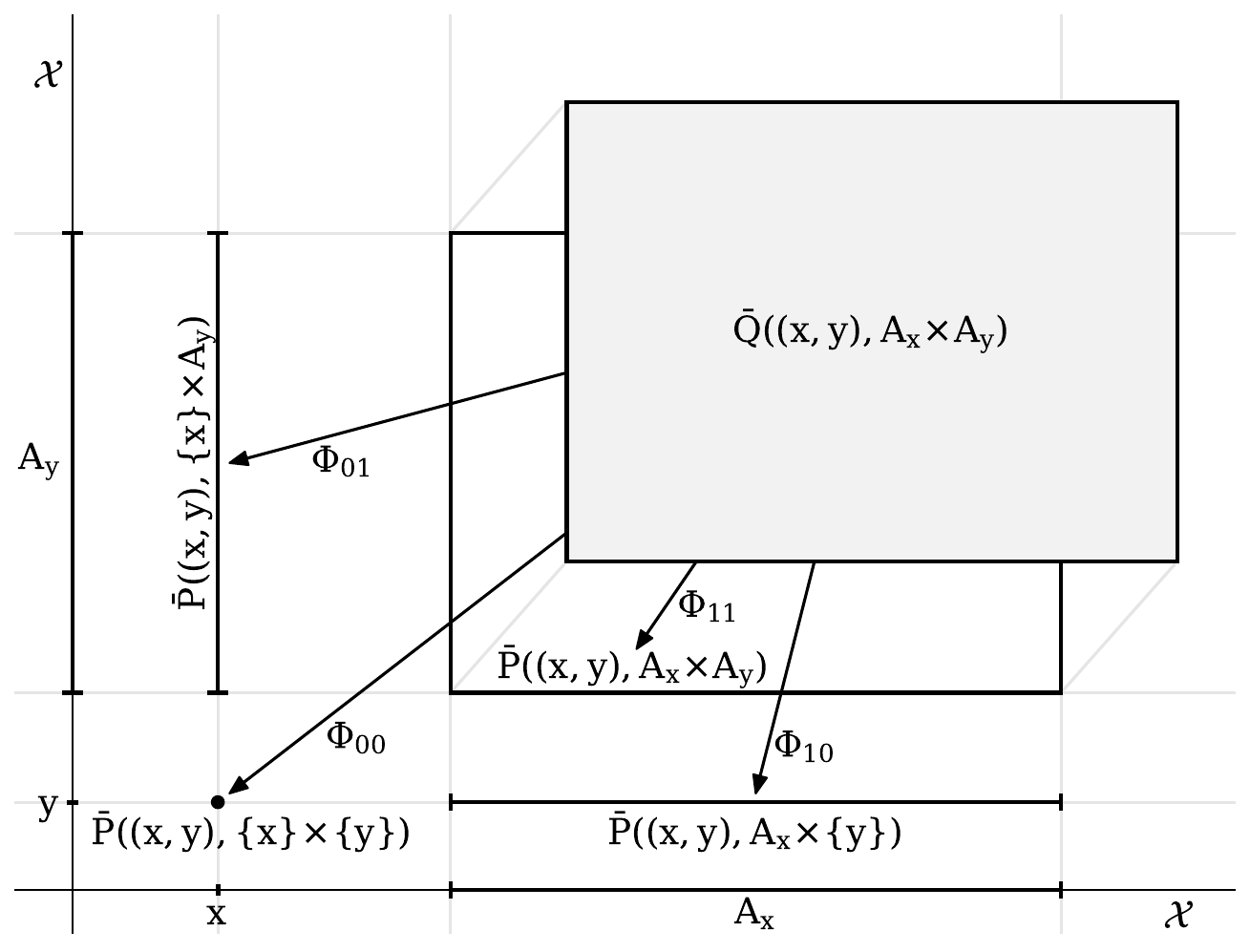}
\caption{
Schematic diagram of a coupled acceptance mechanism $\Phi$ relating a proposal coupling $\bq$ and a
transition kernel coupling $\bp$. Here $(x,y)$ represents the current state and $\axy$ is a
measurable rectangle in $\calXp$. $\bq(\xy, \axy)$ gives the probability of a proposal $\xyp \in
\axy$. The coupled acceptance mechanism $\Phi = (\Phi_{11},\Phi_{10}, \Phi_{01}, \Phi_{00})$
distributes this probability into contributions to the probability $\bp(\xy, \cd)$ of transitions to
the sets $\axy$, $A_x \times \sy$, $\sx \times A_y$, and $\sx \times \sy$. The conditions of
Definition~\ref{def:cam} ensure that $\Phi$ agrees with both $\bq$ and $\bp$. \label{fig:cam}
}
\end{figure}

The crux of Theorem~\ref{thm:repr} is to prove that every coupling of MH-like transition kernels
$\bp$ arises from a proposal and acceptance coupling. Thus, we begin by looking for coupled
acceptance mechanisms $\Phi$ that distribute the probability $\bq(\xy, \axy)$ of a proposal from
$\xy$ to $\xyp \in \axy$ into contributions to the probability of a transition from $\xy$ to $\XY
\in \axy$, $A_x \times \sy, \sx \times A_y$, and $\sx \times \sy$ in a consistent way. The following
result establishes the necessity of the conditions in Definition~\ref{def:cam} for this task.

\smallskip

\begin{lemma}
\label{lem:generation}
Let $\bp \in \Gamma(P,P)$ be a coupling of MH-like transition kernels.
If a proposal coupling $\bq$ and an acceptance indicator coupling $\bb$ generate $\bp$,
then there exists a coupled acceptance mechanism $\Phi$
relating $\bq$ and $\bp$.
\end{lemma}

\begin{proof}
Fix $\xy$, let $(x', y') \sim \bq(\xy, \cdot)$, and let $(b_x, b_y) \sim \bb(\xy, \xyp)$.
For $A \in \scrFp$ and $i, j \in \{0,1\}$, define $\Phi_{ij}(\xy, A) := \P( \xyp \in A, b_x = i, b_y = j \mid x,y)$.
This $\Phi$ satisfies Condition 1, since for all $A_x, A_y \in \scrF$ we have
$\bq(\xy, \axy) = \P(\xyp \in \axy \, | \, x,y)$.
It also satisfies Condition 2, as
\eq{
	& \hspace{-1em} \bp(\xy, \axy) = \P(X \in A_x, Y \in A_y \mid x,y) \\
	& = \P(x' \in A_x, y' \in A_y, b_x=1, b_y=1 \mid x,y)
	+ \P(x' \in A_x, y \in A_y, b_x=1, b_y=0 \mid x,y) \\
	& \quad + \P(x \in A_x, y' \in A_y, b_x=0, b_y=1 \mid x,y)
	+ \P(x \in A_x, y \in A_y, b_x=0, b_y=0 \mid x,y) \\
	& = \Phi_{11}(\xy, \axy)
	+ \Phi_{10}(\xy, A_x \times \calX) 1(y \in A_y) \\
	& \quad + \Phi_{01}(\xy, \calX \times A_y) 1(x \in A_x)
	+ \Phi_{00}(\xy, \calXp) 1(x \in A_x) 1(y \in A_y).
}
Finally, for Condition 3 we have
\eq{
	( \Phi_{11} + \Phi_{10})(\xy, \sx \times \calX) &= \P(x' = x, b_x = 1) = a(x,x) Q(x, \sx) = Q(x,\sx) \\
	( \Phi_{11} + \Phi_{01})(\xy, \calX \times \sy) &= \P(y' = y, b_y = 1) = a(y,y) Q(y, \sy) = Q(y,\sy).
}
Thus $\Phi$ is a coupled acceptance mechanism relating $\bq$ and $\bp$.
\end{proof}

The $\bp$ condition of Definition~\ref{def:cam} takes a more intuitive, if less compact, form when
expressed in terms of measurable sets $A_x, A_y \in \scrF$ with $x \not \in A_x$ and $y \not \in
A_y$:

\smallskip

\begin{lemma}
\label{lem:cases}
Let $\bp \in \Gamma(P,P)$ and suppose $\Phi = (\Phi_{11},\Phi_{10}, \Phi_{01}, \Phi_{00})$ are a collection of measures on $(\calXp, \scrFp)$.
Then Condition 2 of Definition~\ref{def:cam} holds if and only if for all $A_x, A_y \in \scrF$ with $x \not \in A_x$ and $y \not \in A_y$, we have
\eq{
	\mathit 1. \ & \bp(\xy, \axy) = \Phi_{11}(\xy, \axy) \hspace{180pt} \\
	\mathit 2. \ & \bp(\xy, A_x \times \sy) = \Phi_{11}(\xy, A_x \times \sy) + \Phi_{10}(\xy, A_x \times \calX)   \\
	\mathit 3. \ & \bp(\xy, \sx \times A_y) = \Phi_{11}(\xy, \sx \times A_y) + \Phi_{01}(\xy, \calX \times A_y)   \\
	\mathit 4. \ & \bp(\xy, \sx \times \sy) = \Phi_{11}(\xy, \sx \times \sy) + \Phi_{10}(\xy, \sx \times \calX)   \\
	& \hspace{104pt} + \Phi_{01}(\xy, \calX \times \sy) + \Phi_{00}(\xy, \calX \times \calX).}
\end{lemma}

\begin{proof}
The equalities above follow by evaluating Condition 2 of Definition~\ref{def:cam} at the four
possible combinations of $\sx$ or $A_x \sm \sx$ and $\sy$ or $A_y \sm \sy$. For the converse, note that for
any $A_x, A_y \in \scrF$ we have
\eq{
	\bp(\xy, \axy) & = \bp(\xy, \axm \times \aym) + \bp(\xy, \axm
	\times \sy) 1(y \in A_y) \\ & + \bp(\xy, \sx \times \aym) 1(x \in A_x) + \bp(\xy, \sx \times \sy)
	1(x \in A_x) 1(y \in A_y).
}
Replacing these terms with the equalities above yields Condition 2 of Definition~\ref{def:cam}.
\end{proof}

Condition 1 of Lemma~\ref{lem:cases} reflects the fact that the only way for a joint chain to
transition from $\xy$ to $\axy$ with $x \not \in A_x$ and $y \not \in A_y$ is to propose and accept
a transition to some $\xyp \in \axy$. Condition 2 says that the probability of transitioning to a
point in $A_x \times \sy$ amounts to the combined probability of accepting a proposal directly to
such a point or by proposing a move to some $\xyp \in A_x \times \calX$, with $x'$ accepted and $y'$
rejected. The other two conditions have similar interpretations.

%%%

\subsection{Existence of a coupled acceptance mechanism}
\label{sec:camexistence}

To prove Theorem~\ref{thm:repr} we must show that any MH-like transition kernel coupling $\bp \in
\Gamma(P,P)$ arises from a proposal coupling $\bq \in \Gamma(Q,Q)$ and an acceptance indicator
coupling $\bb$ with certain properties. In this subsection we show that for any $\bp \in
\Gamma(P,P)$, there exists a kernel coupling $\bq \in \Gamma(Q,Q)$ and a coupled acceptance
mechanism $\Phi$ relating $\bq$ and $\bp$. In the next subsection we show that we can transform this
$\Phi$ into an acceptance coupling $\bb$. Finally, we use these results to prove our main theorem.
We will make use of the following objects, which are defined for $x \in \calX$ and $A \in \scrF$:
\begin{alignat*}{2}
	\alpha_0(x, A) & := Q(x, A \setminus \sx) - P(x, A \setminus \sx) \quad
	\alpha_1(x, A) && := Q(x, A \cap \sx) + P(x, A \setminus \sx) \\
	\beta(x) & :=\begin{cases}
		\frac{Q(x,\sx)}{P(x,\sx)} & \text{if } P(x,\sx) > 0 \\
		1 & \text{otherwise}
	\end{cases}
	\hspace{18pt}
	\mu(x,A) && := \begin{cases}
		\frac{\alpha_0(x,A)}{\alpha_0(x,\calX)} & \text{if } \alpha_0(x,\calX) > 0 \\
		1(x \in A) & \text{otherwise.}
	\end{cases}
\end{alignat*}
Note that these definitions depend only on the marginal kernels $Q$ and $P$. Each of the above has a simple
interpretation in terms of the underlying chains:

\smallskip

\begin{lemma}
\label{lem:qualitative}
Let $x \in \calX$, $A \in \scrF$, $x' \sim Q(x, \cdot)$, $b_x \sim B(x,x') = \Bern(a(x,x'))$,
and ${X := b_x x' + (1-b_x) x}$,
so that by definition
$X \sim P(x, \cdot)$.
Then $\alpha_i(x, A) = \P(x' \in A, b_x=i \g x)$ for $i \in \zo$, $\beta(x) = \P(b_x = 1 \g x,
X=x)$ when $P(x, \sx) > 0$, and $\mu(x, A) = \P(x' \in A \g b_x = 0, x)$ when
$\alpha_0(x,\calX)>0$.
\end{lemma}

\begin{proof}
By assumption, the acceptance rate function has $a(x,x)=1$ for all $x$.
Thus $\P(x' \in A \cap \sx, b_x = 0 \g x) = 0$, and
\eq{
	\alpha_0(x,A)
	&= Q(x, A \sm \sx) - P(x, A \sm \sx)
	= \P(x' \in A \sm \sx \g x) - \P(X \in A \sm \sx \g x) \\
	&= \P(x' \in A \g x ) - \P(x' \in A, b_x=1 \g x)
	= \P(x' \in A, b_x=0 \g x).
}
Similarly,
\eq{
	& \alpha_1(x,A)
	= Q(x, A \cap \sx) + P(x, A \sm \sx)
	= \P(x' \in A \cap \sx \g x ) + \P(X \in A \sm \sx \g x) \\
	& \quad = \P(x' \in A \cap \sx, b_x=1 \g x ) + \P(x' \in A \sm \sx, b_x=1 \g x)
	= \P(x' \in A, b_x=1 \g x ).
}
Suppose $P(x, \sx) > 0$. Then since $a(x, x) = \P(b_x = 1 \g x, x' = x) = 1$ we have
\eq{
	\beta(x)
	& = \frac{Q(x, \sx)}{P(x, \sx)}
	= \frac{\P(b_x = 1, x'=x \g x ) }{\P(X = x \g x)}
	= \frac{\P(b_x = 1, X=x \g x) }{\P(X = x \g x)}
	= \P(b_x = 1 \g x, X=x).
}
Finally, say $\alpha_0(x, \calX) > 0$. Then
\eq{
	\mu(x,A)
	& = \frac{\alpha_0(x, A)}{\alpha_0(x, \calX)}
	= \frac{ \P(x' \in A, b_x = 0 \g x)}{ \P(b_x = 0 \g x)} = \P(x' \in A \g b_x = 0, x).
}
Thus, each of $\alpha_0, \alpha_1, \beta, \sas \mu$ has a simple meaning in terms of the joint
probability of $x'$ and $b_x$ induced by the proposal kernel $Q$ and the acceptance indicator
distribution $B$.
\end{proof}

We will also need the following properties:

\smallskip

\begin{lemma}
\label{lem:identities}
Let $x \in \calX$ and $A \in \scrF$.
Then
${Q(x, A \cap \sx) \! = \! P(x, A \cap \sx) \beta(x)}$,
${Q(x, A) \! = \! \alpha_1(x, A) + \mu(x, A) \alpha_0(x, \calX)}$,
$\beta(x) \in [0,1]$, and
$\mu(x, \cdot)$ is a measure on $(\calX, \scrF)$,
\end{lemma}

\begin{proof}
When $P(x, \sx) > 0$, $\beta(x) \in [0,1]$ by Lemma~\ref{lem:qualitative}. Otherwise $\beta(x) = 1$,
so in either case $\beta(x) \in [0,1]$. Also by Lemma~\ref{lem:qualitative}, we observe that $\mu(x,
\cdot)$ is the indicator of a measurable set or a well-defined conditional probability, and so in
either case it defines a measure on $(\calX, \scrF)$.

For the first $Q$ equality, if $x \not \in A$, then the $Q(x, A \cap \sx) = 0 = P(x, A \cap \sx)
\beta(x)$. If $x \in A$ then we want to show $Q(x, \sx) = P(x, \sx) \beta(x)$. This holds by the
definition of $\beta(x)$ when $P(x, \sx) > 0$. Since $a(x,x) = 1$ for all $x \in \calX$, $P(x, \sx)
= 0$ implies $Q(x, \sx) = 0$, so the result also holds in this case.

When $\alpha_0(x, \calX) > 0$, the second $Q$ equality holds by the definitions of $\mu, \alpha_0,
\sas \alpha_1$. This leaves the ${\alpha_0(x, \calX)=0}$ case. In general, $\alpha_0(x, \calX) =
Q(x, \sx^c) - P(x, \sx^c) = \int (1 - a(x,x')) Q(x, \diff x')$, since $a(x,x)=1$ for all $x \in
\calX$. Thus $\alpha_0(x, \calX)=0$ implies $a(x,x')=1$ for $Q(x, \cdot)$-almost all $x'$, so
${\alpha_1(x,A)} = {Q(x, A \cap \sx) + \int_{A \setminus \sx} a(x,x') Q(x, \diff x')} = {Q(x, A)}$.
Thus the second equality also holds when $\alpha_0(x, \calX) = 0$.
\end{proof}

We now describe the main lemma used in our proof of Theorem~\ref{thm:repr}.

\smallskip

\begin{lemma} \label{lem:camdpc}
For any coupling $\bp \in \Gamma(P,P)$ of MH-like transition kernels, there exists a proposal
coupling $\bq \in \Gamma(Q,Q)$ and a coupled acceptance mechanism $\Phi$ relating $\bq$ and $\bp$.
\end{lemma}

\begin{proof}
The proof proceeds in three steps. First we define a collection of sub-probability kernels $\Phi =
(\Phi_{11}, \Phi_{10}, \Phi_{01}, \Phi_{00})$ in terms of $\bp$. Then we use $\Phi$ to define a
kernel coupling $\bq \in \Gamma(Q,Q)$. Finally, we show that $\Phi$ satisfies the conditions of
Definition~\ref{def:cam}, making it a coupled acceptance mechanism relating $\bq$ and $\bp$.

We begin with an explicit formula for $\Phi_{11}$ in terms of $\bp$. For $x, y \in \calX$ and $A \in
\scrFp$, let
\eq{
	\Phi_{11}(\xy, A) & := \bp \big( \xy, A \cap (\sx^c \times \sy^c) \big) + \bp \big( \xy, A \cap (\sx \times \sy^c) \big)\,  \beta(x) \\
				     & + \bp \big( \xy, A \cap (\sx^c \times \sy) \big)\,  \beta(y) + \bp \big( \xy, A \cap (\sx \times \sy) \big)\, \beta(x) \, \beta(y).
}
Here $\Phi_{11}$ is a sub-probability kernel, since $\beta(\cdot) \in [0,1]$ by
Lemma~\ref{lem:identities}. We use a product measure construction to define the other three
components of $\Phi$. For $x,y \in \calX$ and $A_x, A_y \in \scrF$, let
\ec{
	\Psi_{10}(\xy, A_x) := \alpha_1(x,A_x) - \Phi_{11}(\xy, A_x \times \calX)
	\qaq
	\Psi_{01}(\xy, A_y) := \alpha_1(y,A_y) - \Phi_{11}(\xy, \calX \times A_y).
}
We claim that $\Psi_{10}(\xy,\cd)$ and $\Psi_{01}(\xy,\cd)$ are sub-probabilities on $(\calX,
\scrF)$. Indeed, by Lemma~\ref{lem:identities} and the fact that $\bp \in \Gamma(P,P)$, we have
\eq{
	\alpha_1(x, A_x)
	&= P(x, A_x \sm \sx) + P(x, A_x \cap \sx) \beta(x) \\
	& = \bp\big((x,y), \axm  \times \sy^c\big)
	+ \bp\big((x,y), \axs  \times \sy^c\big) \beta(x)  \\
	& + \bp\big((x,y), \axm \times \sy\big) +
	\bp\big((x,y), \axs \times \sy\big) \beta(x).
}
Plugging $A_x \times \calX$ into the definition of $\Phi_{11}$ yields analogous terms:
\eq{
	& \Phi_{11}\big((x,y), A_x \times \calX\big)
	= \bp\big((x,y), \axm \times \sy^c\big)
	+ \bp\big((x,y), \axs \times \sy^c\big) \beta(x)  \\
	& \qquad +\bp\big((x,y), \axm \times \sy\big)\beta(y)
	+ \bp\big((x,y), \axs \times \sy\big) \beta(x)\beta(y).
}
Thus $\Psi_{10}(\xy, A_x) = \big( \bp( \xy, \axm \times \sy) + \bp( \xy, \axs \times \sy ) \beta(x)
\big) (1-\beta(y))$, and similarly for $\Psi_{01}(\xy, A_y)$. It follows that $\Psi_{10}(\xy,
\cdot)$ and $\Psi_{01}(\xy, \cdot)$ are sub-probabilities. We also set
\eq{
	\Psi_{00}(x,y) := 1 -
	\alpha_1(x,\calX) - \alpha_1(y, \calX) + \Phi_{11} (\xy, \calXp).
}
Algebraic manipulations show
that $\Psi_{00}(x,y) =  \bp(\xy, \sx \times \sy) (1 - \beta(x)) (1-\beta(y)) \in [0,1]$.

With these results in hand, we define $\Phi_{10}(\xy, \cdot), \Phi_{01}(\xy, \cdot), \sas
\Phi_{00}(\xy, \cdot)$ as the product measures extending the following equalities to $(\calX \times
\calX, \scrF \otimes \scrF)$:
\ec{
	\Phi_{10}(\xy, A_x \times A_y) = \Psi_{10}(\xy,  A_x) \mu(y, A_y) \qquad
	\Phi_{01}(\xy, A_x \times A_y) = \mu(x, A_x) \Psi_{01}(\xy, A_y) \\
	\Phi_{00}(\xy, \axy) = \mu(x, A_x) \mu(y, A_y) \Psi_{00}(x,y).
}
Thus each component of $\Phi = (\Phi_{11}, \Phi_{10}, \Phi_{01}, \Phi_{00})$ is a sub-probability kernel.

Next, we set $\bq(\xy, A) := (\Phi_{11}+\Phi_{10}+\Phi_{01}+\Phi_{00})\big(\xy, A \big)$ for all $A
\in \scrFp$. To see that $\bq \in \Gamma(Q,Q)$, note that for any $x,y \in \calX$ and $A_x \in \scrF$ we have
\eq{
	\bq(\xy, A_x \times \calX)
	& = (\Phi_{11}+\Phi_{10}+\Phi_{01}+\Phi_{00})\big(\xy, A_x \times \calX \big) \\
	& = (\Phi_{11} + \Phi_{10})(A_x \times \calX) + \mu(x, A_x) (\Phi_{01} + \Phi_{00})(\calXp) \\
	& = \alpha_1(x,A_x) + \mu(x, A_x) \alpha_0(x, \calX) = Q(x, A_x).
}
The second-to-last equality follows from the definitions of $\Psi_{10}$, $\Psi_{01}$ and
$\Psi_{00}$, while the last equality is due to Lemma~\ref{lem:identities}. A similar argument shows
that $\bq(\xy, \calX \times A_y) = Q(y, A_y)$ for any $x,y \in \calX$ and $A_y \in \scrF$. Thus $\bq
\in \Gamma(Q,Q)$.

Finally, we show that $\Phi$ satisfies the conditions of Definition~\ref{def:cam}. The $\bq$
condition is automatically satisfied given the definitions above. For the $\bp$ condition, we check
the four cases described in Lemma~\ref{lem:cases}. For the $\axym$ case, we have
$\Phi_{11}(\xy, \axym) = \bp(\xy, \axym)$.
For the $\axm \times \sy$ case,
\eq{
	& \Phi_{11}(\xy, \axm \times \sy) + \Phi_{10}(\xy, \axm \times \calX) \\
	& = P(x, A_x \sm \sx) - \bp(\xy, \axm \times \sy^c)
	= \bp(\xy, \axm \times \sy).
}
Here we have used the definitions of $\Phi_{11}$ and $\Phi_{10}$ and the fact that $\alpha_1(x, A
\sm \sx) = P(x, A \sm \sx)$ for the first equality. The second equality follows from the definition
of $\bp \in \Gamma(P,P)$. By a similar argument
\eq{
	\Phi_{11}(\xy, \sx \times \aym) + \Phi_{10}(\xy, \calX \times \aym) = \bp(\xy, \sx \times \aym).
}
For the $\sx \times \sy$ case we have
\eq{
	& \Phi_{11}(\xy, \sx\! \times\! \sy) + \Phi_{10}(\xy, \sx\! \times\! \calX)
	+ \Phi_{01}(\xy, \calX\! \times\! \sy) + \Phi_{00}(\xy, \calX \! \times \! \calX) \\
	& = \Phi_{11}(\xy, \sx \times \sy) + \Psi_{10}(\xy, \sx) + \Psi_{01}(\xy, \sy) + \Psi_{00}(x,y) \\
	& = 1 + \Phi_{11}(\xy, \sx^c \times \sy^c) - \alpha_1(x, \sx^c) - \alpha_1(y, \sy^c)
	= \bp(\xy, \sx \times \sy).
}
For the last condition of Definition~\ref{def:cam}, note that by the definitions of $\Phi$ and
$\alpha_1$,
\eq{
	(\Phi_{11} + \Phi_{10}) (\xy, \sx \times \calX) & = \alpha_1(x, \sx) = Q(x, \sx) \\
	(\Phi_{11} + \Phi_{01}) (\xy, \calX \times \sy) & = \alpha_1(y, \sy) = Q(y ,\sy).
}
We conclude that $\Phi$ is a coupled acceptance mechanism relating $\bp$ and $\bq$.
\end{proof}

%%%

\subsection{Existence of an acceptance indicator coupling}

Next, we show that if we have a coupled acceptance mechanism $\Phi$ relating $\bq$ and $\bp$, then
there exists an acceptance indicator coupling $\bb$ such that $\bq$ and $\bb$ generate $\bp$. In the
following, we write $\Delta^{n-1}$ for the set of multinomial distributions on a set with $n$
elements.

\smallskip

\begin{lemma}
\label{lem:rnderivs} Say $\bp \in \Gamma(P,P)$, $\bq \in \Gamma(Q,Q)$, and $\Phi$ is a coupled
acceptance mechanism relating $\bp$ and $\bq$. Then there exist $\bq(\xy,\cdot)$-measurable
functions $\phi_{ij}(\xy, \cdot)$ for $i,j \in \{0,1\}$ with $\phi = (\phi_{11}, \phi_{10},
\phi_{01}, \phi_{00})\big(\xy, \xyp \big) \in \Delta^3$ for $\bq(\xy,\cdot)$-almost all $\xyp$. If
we define an acceptance indicator coupling $\bb$ so that $(b_x, b_y) \sim \bb(\xy, \xyp)$ implies
${\P ( b_x = i, b_y = j \g x,y,x',y') = \phi_{ij}\big( \xy, \xyp \big)}$, then $\bq$ and $\bb$
generate $\bp$.
\end{lemma}

\begin{proof}
For each $i,j \in \{0,1\}$, we have $\Phi_{ij}(\xy, \cdot) \ll \bq(\xy, \cdot)$ by the $\bq$
condition of  Definition~\ref{def:cam}. Thus we can form Radon--Nikodym derivatives
$\phi_{ij}(\xy, \cdot) := \mathrm{d} \Phi_{ij}(\xy, \cdot) / \mathrm{d} \bq(\xy, \cdot)$. The RN
derivative is linear and $\mathrm{d} \bq (\xy, \cdot) / \mathrm{d} \bq(\xy, \cdot) = 1$, so for
$\bq(\xy, \cdot)$-almost all $\xyp$ we have $\sum_{i,j \in \{0,1\} } \phi_{ij}(\xy,\xyp) = 1$. Since
$\phi_{ij} \geq 0$, this implies that each $\phi_{ij} \leq 1$. Thus $\phi(\xy, \xyp) \in \Delta^3$
for $\bq(\xy, \cdot)$-almost all $\xyp$.

For points $\xyp$ where the above holds, let $\bxy \sim \bb(\xy, \xyp)$ be the multinomial random
variable on $\{0,1\}^2$ with $\P\big( b_x = i, b_y = j \g x,y,x',y' \big) = \phi_{ij}\big( \xy, \xyp
\big)$ for $i,j \in \{0,1\}$. Suppose that $\xyp \sim \bq(\xy, \cdot)$ and let $X := x' b_x + x
(1-b_x)$ and $Y := y' b_y + y (1-b_y)$. Then for any $A_y, A_y \in \scrF$ we have
\eq{
	\P((X,Y) \in \axy, b_x=1, b_y=1 \g x,y)	&= \Phi_{11}(\xy, \axy) \\
	\P((X,Y) \in \axy, b_x=1, b_y=0 \g x,y) &= \Phi_{10}(\xy, A_x \times \calX) 1(y \in A_y) \\
	\P((X,Y) \in \axy, b_x=0, b_y=1 \g x,y) &= \Phi_{01}(\xy, \calX \times A_y) 1(x \in A_x) \\
	\P((X,Y) \in \axy, b_x=0, b_y= 0 \g x,y) &= \Phi_{00}(\xy, \calX \times \calX) 1(x \in A_x)1(y \in A_y).
}
It follows from these expressions together with Definition~\ref{def:cam} that
${ \P((X,Y) \in \axy) = \bp(\xy, \axy) } $ on all measurable rectangles $\axy$, and hence
$\P((X,Y) \in A) = \bp(\xy, A)$ for all ${A \in \scrFp}$. We conclude that $\bq$ and $\bb$ generate $\bp$.
\end{proof}

Next, we show that when a joint proposal kernel $\bq$ and a joint transition kernel $\bp$ are
related by a coupled acceptance mechanism $\Phi$, then the acceptance indicator coupling derived
above satisfies the conditions of Theorem~\ref{thm:repr}.

\smallskip

\begin{lemma}
\label{lem:rndproperties}
Let $\Phi$ be a coupled acceptance mechanism relating $\bp \in \Gamma(P,P)$ and $\bq \in
\Gamma(Q,Q)$, and let $\phi$ and $\bb$ be as in the proof of Lemma~\ref{lem:rnderivs}. If  $\bxy \!
\sim \! \bb(\xy, \xyp)$, then
\begin{enumerate}
	\item $\P(b_x=1 \g x,y,x') = a(x,x')$ for $Q(x,\cdot)$-almost all $x'$, and
	\item $\P(b_y=1 \g x,y,y')= a(y,y')$ for $Q(y,\cdot)$-almost all $y'$.
\end{enumerate}
\end{lemma}

\begin{proof}
Let $a_x :=  \phi_{11} + \phi_{10}$.
Then $\P(b_x = 1 \g x, y, x') = \E_{\bar Q((x,y),\cdot)}[a_x(\xy, \xyp) \g x, y, x']$,
and $a_x(\xy, \cdot)$ is $\bq(\xy, \cdot)$-measurable.
For all $A \in \scrF$, we have
\eq{
	& \int 1(x' \in A) \P(b_x = 1 \g x, y, x') \bq(\xy, \dxyp)
	= \int 1(x' \in A) a_x(\xy, \xyp) \bq(\xy, \dxyp) \\
	& = (\Phi_{11}+\Phi_{10})(\xy, A \times \calX)
	= Q(x, \sx \cap A) + P(x, A \setminus \sx)
	= \int_A a(x,x') Q(x, \diff x') \\
	& = \int 1(x' \in A) a(x,x') \bq(\xy, \dxyp).
}
The first equality follows from the defining property of conditional expectations. Condition 2 of
Definition~\ref{def:cam} implies ${ (\Phi_{11}+\Phi_{10})(\xy, (A \sm \sx) \times \calX) = \bp(\xy,
	(A \sm \sx) \times \calX) }$, while Condition 3 implies that ${ Q(x, A \cap \sx) } =
{(\Phi_{11}+\Phi_{10})(\xy, (A \cap \sx) \times \calX)}$. These combine to yield the third equality,
above.

By the essential uniqueness of the Radon--Nikodym derivative, $\P(b_x = 1 \g x, y, x') = a(x, x')$
for all $x'$ in a measurable set $\tilde A \in \scrF$ with $\bq(\xy, \tilde A \times \calX) = Q(x,
\tilde A) = 1$. Thus this equality holds for $Q(x,\cdot)$-almost all $x' \in \calX$. A similar
argument shows that $\P(b_y=1 \g x, y, y') = a(y,y')$ for $Q(y,\cdot)$-almost all $y' \in \calX$.
\end{proof}

\subsection{Main result}

Having established the lemmas above, we can now prove the main result of this section.

\begin{proof}[Proof of Theorem~\ref{thm:repr}]
For the `if' case, assume $\bq \in \Gamma(Q,Q)$ and $\bxy \sim$ $\bb(\xy,\xyp)$, an acceptance
indicator coupling with $\P(b_x=1 \g x,y,x') = a(x,x')$ for $Q(x,\cdot)$-almost all $x'$ and
$\P(b_y=1 \g x,y,y')= a(y,y')$ for $Q(y,\cdot)$-almost all $y'$. Let $\bp(\xy, \cdot)$ be the law of
$(X,Y) := (b_x x' + (1-b_x) x, b_y y' + (1 - b_y) y).$ Then for any $A \in \scrF$,
\eq{
	& \bp(\xy, A \times \calX)
	= \P(X \in A \g x,y) = \P(x' \in A, b_x=1 \g x,y) + \P(x \in A, b_x=0 \g x,y) \\
	& = \int 1(x' \in A) a(x,x') Q(x, \diff x') + 1(x \in A) \int (1-a(x,x')) Q(x, \diff x')
	= P(x, A).
}
Similarly, $\bp(\xy, \calX \times A) = P(y, A)$ for any $A \in \scrF$. Thus $\bp
\in \Gamma(P,P)$.

For the `only if' case, take any $\bp \in \Gamma(P,P)$. By Lemma~\ref{lem:camdpc}, there exists a
proposal coupling $\bq \in \Gamma(Q,Q)$ and a coupled acceptance mechanism $\Phi$ relating $\bq$ and
$\bp$. Then by Lemma~\ref{lem:rnderivs}, there exists an acceptance indicator coupling $\bb$ such
that $\bq$ and $\bb$ generate $\bp$. Finally, by Lemma~\ref{lem:rndproperties}, this $\bb$ will be
such that if $\bxy \sim \bb(\xy,\xyp)$, then $\P(b_x=1 \g x,y,x') = a(x,x')$ for $Q(x,\cdot)$-almost
all $x'$ and $\P(b_y=1 \g x,y,y')= a(y,y')$ for $Q(y,\cdot)$-almost all $y'$. Thus we conclude that
$\bp$ is generated by a proposal coupling $\bq$ and an acceptance indicator coupling $\bb$ with the
desired properties.
\end{proof}

\subsection{Simplified characterization of couplings}
\label{sec:cor}

To prove the `only if' part of Theorem~\ref{thm:repr} we constructed a map from an arbitrary
coupling $\bp \in \Gamma(P,P)$ to a proposal coupling $\bq \in \Gamma(Q,Q)$ and an acceptance
indicator coupling $\bb$ such that $\bq$ and $\bb$ generate $\bp$.  Lemmas
\ref{lem:qualitative}--\ref{lem:rndproperties} suggest a potentially complicated relationship
between $\bp$, $\bq$, and $\bb$, at least when $P$ is allowed to be any MH-like transition kernel.
As we show below in Corollary~\ref{cor:repr_non_atomic}, this relationship takes a more intuitive
form if we require $Q(x, \cdot)$ to be absolutely continuous with respect to the base measure and
require $Q(x, \sx) = 0$ for all $x \in \scrF$. These assumptions often hold, for instance when the state space
is continuous and $\lambda$ is non-atomic, when $Q$ represents a non-lazy random walk on a discrete
state space, and in many other cases.

We need to make a few definitions and analytical observations before stating the main result of this
section. First, for all $A_x \in \scrF$ let ${\bp_y(x, A_x)} := {\bp(\xy, (A_x \sm \sx) \times
	\sy)}$, so that $\bp_y(x, \cdot)$ is a sub-probability on $(\calX, \scrF)$. Then for all $x, y \in
\calX$, ${Q(x, \cdot) \ll \lambda}$ implies ${\bp_y(x, \cdot) \ll \lambda}$, since \eq{ \bp_y(x,
	A_x) & = \bp(\xy, (A_x \sm \sx) \times \sy) \leq \bp(\xy, (A_x \sm \sx) \times \calX) \\ & = P(x,
	A_x \sm \sx) = \int_{A_x \sm \sx} a(x,x') q(x,x') \lambda(\diff x'). } Here $q(x,\cdot) = \diff
Q(x,\cdot) / \diff \lambda$. Therefore $\bp_y(x, \cdot)$ will have density $\bar p_y(x, \cdot)$ with
respect to $\lambda$ by the Radon--Nikodym theorem. Similarly, $\bp_x(y, A_y) := \bp(\xy, \sx \times
(A_y \sm \sy))$ will have density $\bar p_x(y, \cdot)$ with respect to $\lambda$.

Let $r(x) := P(x, \sx)$, the probability of a transition from $x$ to itself under $P$. From the
definitions in Section~\ref{sec:camexistence}, we have $\beta(x) = 1(r(x) = 0)$ and
\eq{
	\mu(x, A_x) := \begin{cases}
		\frac{ Q(x, A_x) - P(x, A_x \sm \sx) } {r(x)} & \mbox{if } r(x) > 0 \\
		1(x \in A_x) & \mbox{if } r(x) = 0.
	\end{cases}
}
It follows that $\mu(x, A_x) (1 - \beta(x)) = (Q(x, A_x) - P(x, A_x \sm \sx))/r(x)$ if $r(x) >0$,
and 0 otherwise. Thus ${\mu(x, \cdot) (1 - \beta(x))}$ has density $m(x,x') := q(x,x')
(1-a(x,x'))/r(x)$ when $r(x) > 0$, and 0 otherwise.

Next, we separate $\bp$ into a part that is absolutely continuous with respect to $\lambda \times
\lambda$ and another part that is singular to it. By the Lebesgue decomposition theorem
\citep[e.g.][chapter 5.5]{Dudley2002} for each $x, y \in \calX$ we have ${\bp(\xy, \cdot)} =
{\bp_\ll(\xy, \cdot) + \bp_\perp(\xy, \cdot)}$ with $\bp_\ll(\xy, \cdot) \ll \lambda \times \lambda$
and $\bp_\perp(\xy, \cdot) \perp \lambda \times \lambda$. We then define $p(\xy, \cdot)$ to be the
density of $\bp_\ll(\xy, \cdot)$ with respect to $\lambda \times \lambda$. Finally, let $\bar r(x,y)
:= \bp(\xy, \sx \times \sy)$, the probability of a transition from $\xy$ to itself under $\bp$. With
these definitions, we can now give an explicit characterization of the couplings $\bp \in
\Gamma(P,P)$ under mild conditions on the proposal distribution.

\smallskip

\begin{corollary}
\label{cor:repr_non_atomic} Let $P$ be the MH-like transition kernel on $(\calX, \scrF)$ generated
by a proposal kernel $Q$ and an acceptance rate function $a$, and suppose that $Q(x, \cdot) \ll
\lambda$ with $Q(x,\sx) = 0$ for all $x \in \calX$. If a joint kernel $\bp \in \Gamma(P,P)$, then it
is generated by a proposal coupling $\bq = \bq_\ll + \bq_\perp \in \Gamma(Q,Q)$ and an acceptance
indicator coupling $\bb$ with the following properties for all $x,y \in \calX$ and $\bxy \sim
\bb(\xy, \xyp)$:
\begin{enumerate}
	\item $\bq_\perp(\xy, A) = \bp_\perp(\xy, A \cap (\sx^c \times \sy^c))$ for $A \in \scrF$,
	\item $\bq_\ll(\xy, \cdot)$ has the following density with respect to $\lambda \times \lambda$:
	\eq{
		q(\xy, \xyp)
		= p(\xy, \xyp) + \bar p_y(x, x') m(y,y')
		+ m(x,x') \bar p_x(y, y') + m(x,x') m(y, y') \bar r(x,y),
	}
	\item $b_x = b_y = 1$ almost surely when $\xyp$ is in the support of $\bq_\perp(\xy, \cdot)$,
\end{enumerate}
and for $\bq_\ll(\xy, \cdot)$-almost all $\xyp$,
\begin{enumerate}
	\setcounter{enumi}{3}
	\item $\P(b_x = 1 \g x,y,x',y')\, q(\xy, \xyp) = p(\xy, \xyp) + \bar p_y(x, x') m(y,y')$,
	\item $\P(b_y = 1 \g x,y,x',y')\, q(\xy, \xyp) = p(\xy, \xyp) + \bar p_x(y, y') m(x,x')$.
\end{enumerate}
\end{corollary}

\begin{proof}
Let $B_{xy} := \sx^c \times \sy^c$.
We may focus attention on the behavior of $\bq(\xy, \cdot)$ on $B_{xy}$,
since ${\bq(\xy, B_{xy}) = 1}$. This follows since for any $x,y \in \calX$,
$B_{xy}^c = (\sx \times \calX) \cup (\calX \times \sy)$,
$\bq(\xy, \sx \times \calX) = Q(x, \sx) = 0$, and
$\bq(\xy, \calX \times \sy) = Q(y, \sy) = 0$.

When $\bp \in \Gamma(P,P)$, Lemma~\ref{lem:camdpc} tells us that there exists a proposal
coupling $\bq \in \Gamma(Q,Q)$ and a coupled acceptance mechanism $\Phi$ relating $\bp$ and $\bq$.
In particular, for $x,y \in \calX$ and $A_x, A_y \in \scrF$ with $x \not \in A_x$ and $y \not
\in A_y$, the proof of that Lemma shows that we can use the following measures:
\ec{
	\bq(\xy, \axy) = (\Phi_{11} + \Phi_{10} + \Phi_{01} + \Phi_{00})(\xy, \axy) \\
	\Phi_{11}(\xy, \axy) = \bp ( \xy, \axy ) \\
	\Phi_{10}(\xy, \axy) = \bp( \xy, A_x \times \sy) (1 - \beta(y))  \mu(y, A_y) \\
	\Phi_{01}(\xy, \axy) = (1 - \beta(x)) \mu(x, A_x) \bp( \xy, \sx \times A_y) \\
	\Phi_{00}(\xy, \axy) = (1 - \beta(x)) \mu(x, A_x) (1-\beta(y)) \mu(y, A_y) \bp(\xy, \sx \times \sy).
}
From the Lebesgue decomposition of $\bp$, we have
\eq{
	\bp ( \xy, \axy ) = \bp_\perp(\xy,\axy) + \int_{\axy} p(\xy,\xyp) \diff \xyp.
}
The definitions above imply that $\Phi_{10}(\xy, \cdot)$ has density $p_y(x, x') m(y,y')$,
$\Phi_{01}(\xy, \cdot)$ has density $p_x(y, y') m(x,x')$, and $\Phi_{00}(\xy, \cdot)$ has density
$\bar r(x,y) m(x,x') m(y,y')$. Thus $\bq(\xy, \cdot)$ is the sum of a singular part ${\bq_\perp(\xy,
	A)} = {\bp(\xy, A \cap B_{xy})}$ and an absolutely continuous part $\bq_\ll(\xy, \cdot)$ with the
density function specified in Condition 2, above.

Next by Lemma~\ref{lem:rnderivs}, we know that there exists an acceptance indicator coupling $\bb$
such that $\bq$ and $\bb$ generate $\bp$. This $\bb$ has the property that $\bq(\xy, \cdot)$-almost
all values of $\xyp$, if $\bxy \sim \bb(\xy,\xyp)$, then
\eq{
	\P(b_x=1 \g x,y,x',y') &= \diff (\Phi_{11} + \Phi_{10})(\xy, \cdot) / \diff \bq(\xy, \cdot) \\
	\P(b_y=1 \g x,y,x',y') & = \diff (\Phi_{11} + \Phi_{01})(\xy, \cdot) / \diff \bq(\xy, \cdot).
}
For $A \in \scrF$ contained in the support of $\bq_\perp(\xy, \cdot)$, we have $\Phi_{11}(\xy, A) =
\bq(\xy, A)$ and ${\Phi_{10}(\xy, A)} = {\Phi_{01}(\xy, A) = 0}$, so the Radon--Nikodym derivatives
above equal 1 almost surely. This proves Condition 3. Otherwise, for $\bq(\xy, \cdot)$-almost all
points $\xyp$ with $q(\xy, \xyp) > 0$ we have
\eq{
	\P(b_x=1 \g x,y,x',y')
	& = \frac{p(\xy, \xyp) + \bar p_y(x, x') m(y,y')}{q(\xy,\xyp)}
}
and similarly for $\P(b_y \g x,y,x',y')$.
Finally, $p(\xy, \xyp) = \bar p_y(x,x') = \bar p_x(y,y') = 0$ whenever ${q(\xy, \xyp) = 0}$,
thus proving Conditions 4 and 5.
\end{proof}

Corollary \ref{cor:repr_non_atomic} applies to MH-like algorithms on both continuous and discrete
spaces, as long as $Q(x,\cdot)$ has a density or mass function and as long as the proposal $x'$
almost surely differs from the current point $x$. It provides an explicit way to represent any
kernel coupling in terms of proposal and acceptance indicator couplings.

%%%%%%%%%%%%%%%%%%%%%%%%%%%%%%

\section{Maximal kernel couplings}
\label{sec:maximal}

We now apply the results above to characterize the maximal couplings of a general MH-like transition
kernel $P$. The identification of couplings $\bp \in \Gamma(P,P)$ that induce fast meeting between
chains is an important question for both theoretical analysis and in applied work. Maximal couplings
$\bp \in \Gmax(P,P)$ represent myopically optimal solutions to this problem, in the sense that they
achieve the highest one-step meeting probability $\P(X=Y \g x, y)$ from each state pair $(x,y)$.
Understanding the structure of $\Gmax(P,P)$ can aid in coupling design and serves as a reference
point in the search for efficient Markovian couplings of MH-like chains.

The literature so far has said little about $\Gmax(P,P)$.~\citet{OLeary2020} showed that when the
state space $\calX$ is continuous, this set contains at least a few computationally feasible
couplings. That paper also showed that for any maximal proposal coupling $\bq \in \Gmax(Q,Q)$, there
exists an acceptance indicator coupling $\bb$ such that the resulting transition kernel coupling
$\bp \in \Gmax(P,P)$. Maximal proposal couplings are often easy to construct, so this observation
provides a way to obtain a range of couplings in $\Gmax(P,P)$. It also suggests a closer
relationship between $\Gmax(Q,Q)$ and $\Gmax(P,P)$ than exists in general, as we discuss below.

In the following analysis we write $\Delta := \{(z,z) : z \in \calX\}$ for the diagonal of $\calXp$,
$\delta : \calX \to \Delta$ for the map ${z \mapsto (z,z)}$, and $A_\Delta := \delta(A) = \{(z,z) :
z \in A\}$. We continue to assume that $(\calX, \scrF)$ is a Polish space, so that $\Delta \in
\scrFp$ and so that $\delta$ is a measurable function. As noted above, the coupling inequality states that
if $\mu$ and $\nu$ are probability measures on $(\calX, \scrF)$ and if ${\gamma \in \Gamma(\mu,
	\nu)}$, then $\P_{\XY \sim \gamma}(X = Y) \leq 1 - \lVert \mu - \nu \rVert_\TV$. A maximal coupling
is one that achieves this bound. In the following section we see that this is equivalent to a
measure-theoretic condition, which we use to characterize maximal couplings of MH-like kernels in
Section~\ref{sec:maxc}.

\subsection{The Hahn maximality condition}

Given probability measures $\mu$ and $\nu$ on $(\calX, \scrF)$, the Hahn-Jordan theorem
\citep[e.g.][chapter 5.6]{Dudley2002} states that there exists a measurable set $S \in \scrF$ and
sub-probability measures $\mu^r$ and $\nu^r$ such that $\mu - \nu = \mu^r - \nu^r$ and $\mu^r(S^c) =
\nu^r(S) = 0$. The pair $(S,S^c)$ is called a Hahn decomposition for $\mu - \nu$, and it is
essentially unique in the sense that if $R \in \scrF$ is another measurable set with $\mu^r(R^c) =
\nu^r(R) = 0$, then $(\mu-\nu)(S \symdiff R) = 0$. Here $A \symdiff B = (A \sm B) \cup (B \sm A)$
indicates the symmetric difference of measurable sets $A, B \in \scrF$. The pair $(\mu^r, \nu^r)$ is
called the Jordan decomposition of $\mu - \nu$, and it is uniquely characterized by the above.

In the Jordan decomposition $\mu - \nu = \mu^r - \nu^r$, $\mu^r$ and $\nu^r$ are called the upper
and lower variation of $\mu - \nu$, and ${\mu \wedge \nu := \mu - \mu^r = \nu - \nu^r}$ is called
the meet or infimum measure of $\mu$ and $\nu$. Here $\mu \wedge \nu$ is non-negative and has the
defining property that if $\eta$ is another measure on $(\calX, \scrF)$ with $\eta(A) \leq \mu(A)
\wedge \nu(A)$ for all $A \in \scrF$, then $\eta(A) \leq (\mu \wedge \nu)(A)$ for all $A \in \scrF$.
By the definition of total variation, $\lVert \mu - \nu \rVert_\TV = \sup_{A \in \scrF}\lvert \mu(A)
- \nu(A) \rvert = \mu^r(\calX) = \nu^r(\calX) = 1 - (\mu \wedge \nu) (\calX)$. See e.g.
\citet[][chap. 5]{dshalalow2012foundations} or \citet[][sec. 36]{aliprantis1998principles} for more
on the lattice-theoretic properties of the set of measures on $(\calX, \scrF)$.

For any measure $\mu$ on $(\calX, \scrF)$ let $\delta_\star \mu$ be the pushforward of $\mu$ by the
diagonal map $\delta$, so that $\delta_\star \mu (A) = \mu (\delta^{-1} (A))$ for $A \in \scrFp$.
This makes $\delta_\star \mu$ a measure on $(\calXp, \scrFp)$ with support contained in $\Delta$ and
with $\delta_\star \mu (B_\Delta) = \mu (B)$ for any $B \in \scrF$. In this notation, maximal
couplings can be characterized as follows:

\smallskip

\begin{lemma}[\citet{douc2018markov}, Theorem 19.1.6]
\label{lem:douc}
Let $\mu - \nu = \mu^r - \nu^r$ be the Jordan decomposition
of a pair of probability measures $\mu$ and $\nu$ on $(\calX, \scrF)$.
A coupling $\gamma \in \Gamma(\mu, \nu)$ is maximal if and only if there exists a
$\gamma^r  \in \Gamma(\mu^r, \nu^r )$
such that $\gamma(A) = \gamma^r(A) + \delta_\star \big( \mu \wedge \nu \big) (A)$ for all $A \in \scrFp$.
\end{lemma}

Note that we must have $\mu^r(\calX) = \nu^r(\calX)$ for $\Gamma(\mu^r,\nu^r)$ to be
nonempty. This follows from the Jordan decomposition, since $0 = (\mu - \nu)(\calX) = (\mu^r -
\nu^r)(\calX)$. Also $\gamma^r(\Delta) = 0$, as
\eq{
	1 -\lVert \mu - \nu \rVert_\TV
	= \gamma(\Delta)
	= \gamma^r(\Delta) + \delta_\star(\mu \wedge \nu)(\Delta)
	= \gamma^r(\Delta) + (\mu \wedge \nu)(\calX)
	= \gamma^r(\Delta) + 1 -\lVert \mu - \nu \rVert_\TV.
}
Finally, we observe that Lemma~\ref{lem:douc} implies the maximal coupling recognition result of
\citet[Lemma 20]{Ernst2019}. Thus we have the following characterization of maximal couplings based
on the Hahn decomposition:

\smallskip

\begin{corollary}[Hahn Maximality Condition]
\label{cor:recog} Let $\mu$ and $\nu$ be measures on $(\calX, \scrF$). A coupling $\gamma \in
\Gamma(\mu, \nu)$ is maximal if and only if there is an $S \in \scrF$ with $\gamma( (S^c \times
\calX) \sm \Delta )$ $= \gamma( (\calX \times S) \sm \Delta ) = 0$. Any $(S,S^c)$ with this property
will be a Hahn decomposition for $\mu-\nu$.
\end{corollary}

\begin{proof}
Let $\mu - \nu = \mu^r - \nu^r$ be a Jordan decomposition, so that we have $\mu^r(S^c) = \nu^r(S) =
0$ for some $S \in \scrF$. If ${\gamma \in \Gamma(\mu,\nu)}$ is maximal, Lemma~\ref{lem:douc}
implies that $\gamma(A) = \gamma^r(A) + \delta_\star(\mu \wedge \nu)(A)$ for all $A \in \scrFp$.
Therefore ${\gamma((S^c \times \calX) \sm \Delta)} = {\gamma^r(S^c \times \calX)} = {\mu^r(S^c) =
	0}.$ By a similar argument, $\gamma( (\calX \times S) \sm \Delta ) = 0$. For the converse, let
${\gamma \in \Gamma(\mu,\nu)}$ and ${\gamma( (S^c \times \calX) \sm \Delta )} = {\gamma( (\calX
	\times S) \sm \Delta ) = 0}$. Then $\mu(B) = \gamma(B \times \calX) = \gamma( (B \times \calX) \sm
\Delta) + \gamma(B_\Delta)$ and $\nu(B) = \gamma(\calX \times B) = \gamma( (\calX \times B) \sm
\Delta) + \gamma(B_\Delta)$ for any $B \in \scrF$.

By assumption, $S$ contains the support of $\gamma((\cd \times \calX) \sm \Delta)$ and $S^c$
contains the support of $\gamma((\calX \times \cd)\sm\Delta )$. Thus ${ \mu(\cd) - \nu(\cd) =
	\gamma((\cd \times \calX) \sm \Delta) - \gamma((\calX \times \cd)\sm\Delta)}$ is the Jordan
decomposition of $\mu - \nu$ and $(S,S^c)$ is a Hahn decomposition. The uniqueness of the Jordan
decomposition implies $\gamma((B \times \calX) \sm \Delta) = \mu^r(B)$ and $\gamma((\calX \times B)
\sm \Delta) = \nu^r(B)$, which in turn yields $\gamma^r(\cd) := \gamma(\cd \sm \Delta) \in
\Gamma(\mu^r, \nu^r)$. We also have $(\mu \wedge \nu)(B) = \mu(B) - \mu^r(B) = \nu(B) - \nu^r(B)$ ,
so the above implies $\gamma(B_\Delta) = (\mu \wedge \nu)(B)$ for all $B \in \scrF$. We conclude
that $\gamma(A) = \gamma(A \sm \Delta) + \gamma(A \cap \Delta)= \gamma^r(A) + \delta_\star(\mu
\wedge \nu)(A)$ for any $A \in \scrFp$.
\end{proof}

In Section~\ref{sec:maxc}, we use this result to establish conditions for $\bp \in \Gmax(P,P)$ in
terms of proposal and acceptance couplings. Before doing so, we need to clarify the relationship
between the maximality of a transition kernel coupling $\bp$ and the maximality of a proposal
coupling $\bq$ that generates it.

\subsection{Maximal proposal kernel couplings}
\label{sec:maxprop}

It seems reasonable to guess that if $\bp \in \Gmax(P,P)$ is generated by a proposal coupling $\bq
\in \Gamma(Q,Q)$ and an acceptance indicator coupling $\bb$, then $\bq$ must itself be maximal. The
proposal-based maximal couplings of~\citet{OLeary2020} have this property, and it seems plausible
that in order to maximize the probability of $X=Y$ one might need to start by maximizing the
probability of $x'=y'$. However, the following shows that no special relationship exists between
maximal proposal and transition couplings.

\smallskip

\begin{lemma}
\label{lem:maxtonon} Suppose the transition kernel coupling $\bp \in \Gmax(P,P)$ is generated by a
coupling ${\bq_m \in \Gmax(Q,Q)}$ and an acceptance coupling $\bb_m$. Suppose that $\bq_m(\xy,
\Delta^c) > 0$ for some $\xy$, and at that $\xy$, ${\P(b_x = b_y = 1 \g x,y) < 1}$ where
$(x_m',y_m') \sim \bq_m(\xy,\cd)$ and $(b_x,b_y) \sim \bb_m(\xy, (x_m',y_m'))$. Then there exists a
non-maximal coupling $\bq \in \Gamma(Q,Q)$ and an acceptance indicator coupling $\bb$ such that
$\bq$ and $\bb$ also generate $\bp$.
\end{lemma}

In the following proof, we use $\bq_m$ to construct a $\bq$ that agrees with $\bq_m$ on accepted
proposals and independently redraws rejected ones. The hypotheses on $\bq_m$ and on $\bxy$ are
needed to ensure that the resulting $\bq$ is not maximal.

\begin{algorithm}
	\caption{ Construction of $\bq$ for Lemma~\ref{lem:maxtonon} \label{alg:nonmax}}
	\normalsize
	\begin{enumerate}
		\item Draw $(x'_m, y'_m) \sim \bq_m(\xy, \cd)$ and $(b_x, b_y) \sim \bb_m(\xy, (x'_m, y'_m))$
		\item For $z \in \{x, y\}$:
		\begin{enumerate}
			\item If $b_z = 1$, set $z' = z'_m$. Else:
			\item Draw $(\tilde x, \tilde y) \sim \bq_m(\xy, \cd)$
			and $(\tilde b_x, \tilde b_y) \sim \bb_m(\xy, (\tilde x, \tilde y)$
			\item If $(\tilde b_x, \tilde b_y) = (b_x, b_y)$, set $z' = \tilde z$
			\item Else: go to 2(b)
		\end{enumerate}
		\item Return $(x', y')$ and $(b_x, b_y)$
	\end{enumerate}
\end{algorithm}

\begin{proof}
Let $\bq(\xy,\cd)$ be the distribution of the $\xyp$ output of Algorithm~\ref{alg:nonmax}.
We claim that $\bq \in \Gamma(Q,Q)$. For $A \in \scrF$,
\eq{
	\bq(\xy, A \times \calX)
	& = \P(x_m' \in A, b_x=1 \g x,y) + \hspace{-7pt} \sum_{j \in \{0,1\} } \P(b_x=0, b_y=j \g x,y) \P(\tilde x \in A \g \tilde b_x=0, \tilde b_y=j, x,y) \\
	& = \P(x_m' \in A, b_x=1 \g x,y) + \hspace{-7pt} \sum_{j \in \{0,1\} } \P(x_m' \in A, b_x=0, b_y=j \g x,y)
	= \P(x'_m \in A \g x) = Q(x,A).
}
Similarly, $\bq(\xy, \calX \times A) = Q(y,A)$. Let $\xy$ be such that $\bq_m(\xy, \Delta^c) > 0$.
We have assumed that there is a positive probability of rejecting either $x'_m$ or $y'_m$ at $\xy$.
Thus $\P(x' = y' \g x,y) < \P(x'_m = y'_m \g x,y)$, and so we conclude that $\bq$ is not a maximal
coupling.

Next, for $i,j \in \{0,1\}$ and $A \in \scrFp$, define $\Phi_{ij}(\xy, A) := \P((x',y') \in A, b_x =
i, b_y=j)$ using the full output of Algorithm~\ref{alg:nonmax}. We observe that this is a coupled
acceptance mechanism relating $\bq$ and $\bp$. The first condition of Definition~\ref{def:cam} is
satisfied by construction. For the second condition, define $X_m = b_x x_m' + (1-b_x) x$ and $Y_m =
b_y y_m' + (1-b_y) y$. Since $\bq_m$ and $\bb_m$ generate $\bp$, we must have $(X_m, Y_m) \sim
\bp(\xy,\cd)$. Thus for any $\axy \in \scrFp$,
\eq{
	&  \Phi_{11}(\xy, \axy) + \Phi_{10}(\xy, A_x \times \calX) 1(y \in A_y)  \\
	& \qquad + \Phi_{01}(\xy, \calX \times A_y) 1(x \in A_x)
	+ \Phi_{00}(\xy, \calX \times \calX) 1(x \in A_x) 1(y \in A_y) \\
	& = \P( X_0 \in A_x, Y_0 \in A_y, b_x = 1, b_y=1 \g x,y)
	+ \P( X_0 \in A_x, Y_0 = y, b_x = 1, b_y=0 \g x,y) \\
	& \qquad + \P( X_0 = x, Y_0 \in A_y, b_x = 0, b_y=1 \g x,y)
	+ \P( X_0 = x ,Y_0 = y, b_x = 0, b_y=0 \g x,y)
	= \bp(\xy, \axy).
}
The third condition on $\Phi$ follows from the fact that $b_x = 1$ if $x_m'=x$ and $b_y = 1$ if
$y_m'=y$. Since $\Phi$ is a coupled acceptance mechanism relating $\bq$ and $\bp$,
Lemma~\ref{lem:rnderivs} ensures the existence of a $\bb$ such that $\bq$ and $\bb$ generate $\bp$.
\end{proof}

Lemma \ref{lem:maxtonon} shows that if a maximal transition kernel coupling $\bp \in \Gmax(P,P)$ is
generated by a proposal coupling $\bq$ and an acceptance coupling $\bb$, then $\bq$ need not be
maximal. In Appendix~\ref{ex:nonmax} we observe that some maximal couplings $\bp$ cannot be
generated from any maximal coupling $\bq \in \Gmax(Q,Q)$. Although~\citet{OLeary2020} showed that we
can usually derive \textit{some} maximal coupling $\bp$ from a maximal coupling $\bq$, we conclude
that there is no general relationship between the maximality of a proposal kernel coupling and
that of an associated transition kernel coupling.

\subsection{Characterization of maximal kernel couplings}
\label{sec:maxc}

We now turn to the main result of this section, which extends Theorem~\ref{thm:repr} to characterize
the maximal couplings of an MH-like transition kernel in terms of proposal and acceptance indicator
couplings. For each $x,y \in \calX$, let $(\sxy, \sxy^c)$ be a Hahn decomposition for $P(x,\cd) -
P(y, \cd)$. Thus $\sxy \in \scrF$ and $P(x,A) \geq P(y,A)$ for any $A \in \scrF$ with $A \subset
\sxy$. Note that if $P(x,\cd)$ and $P(y,\cd)$ have Radon--Nikodym derivatives $p(x,\cd)$ and
$p(y,\cd)$ with respect to a common dominating measure, then we can use $S_{xy} := \{z : p(x, z)
\geq p(y,z) \}$ to form these sets.

\smallskip

\begin{theorem}
\label{thm:max}
Let $P$ be the MH-like transition kernel on $(\calX, \scrF)$ generated by a proposal kernel $Q$ and
an acceptance rate function $a$. Then $\bp \in \Gmax(P,P)$ if and only if $\bp$ is generated by
$\bq \in \Gamma(Q,Q)$ and an acceptance indicator coupling $\bb$ with the following properties: if
$\bxy \sim \bb(\xy,\xyp)$, then for all $x, y \in \calX$:
\begin{enumerate}
	\item $\P(b_x=1 \g x,y,x') = a(x,x')$ for $Q(x,\cdot)$-almost all $x'$
	\item $\P(b_y=1 \g x,y,y')= a(y,y')$ for $Q(y,\cdot)$-almost all $y'$
\end{enumerate}
and for $\bq(\xy,\cd)$-almost all $\xyp$,
\begin{enumerate}
	\setcounter{enumi}{2}
	\item $\P(b_x=b_y=1\g x,y,x',y') = 0$ if $x' \neq y'$ and either $x' \in \sxy^c$ or $y' \in \sxy$
	\item $\P(b_x=1, b_y=0\g x,y,x',y') = 0$ if $x' \neq y$ and either $x' \in \sxy^c$ or $y \in \sxy$
	\item $\P(b_x=0, b_y=1 \g x,y,x',y') = 0$ if $y' \neq x$ and either $x \in \sxy^c$ or $y' \in \sxy$
	\item $\P(b_x=0, b_y=0 \g x,y,x',y') = 0$ if $x \neq y$ and either $x \in \sxy^c$ or $y \in \sxy$.
\end{enumerate}
\end{theorem}

\begin{figure}[t]
\centering
\includegraphics[width=.58 \linewidth]{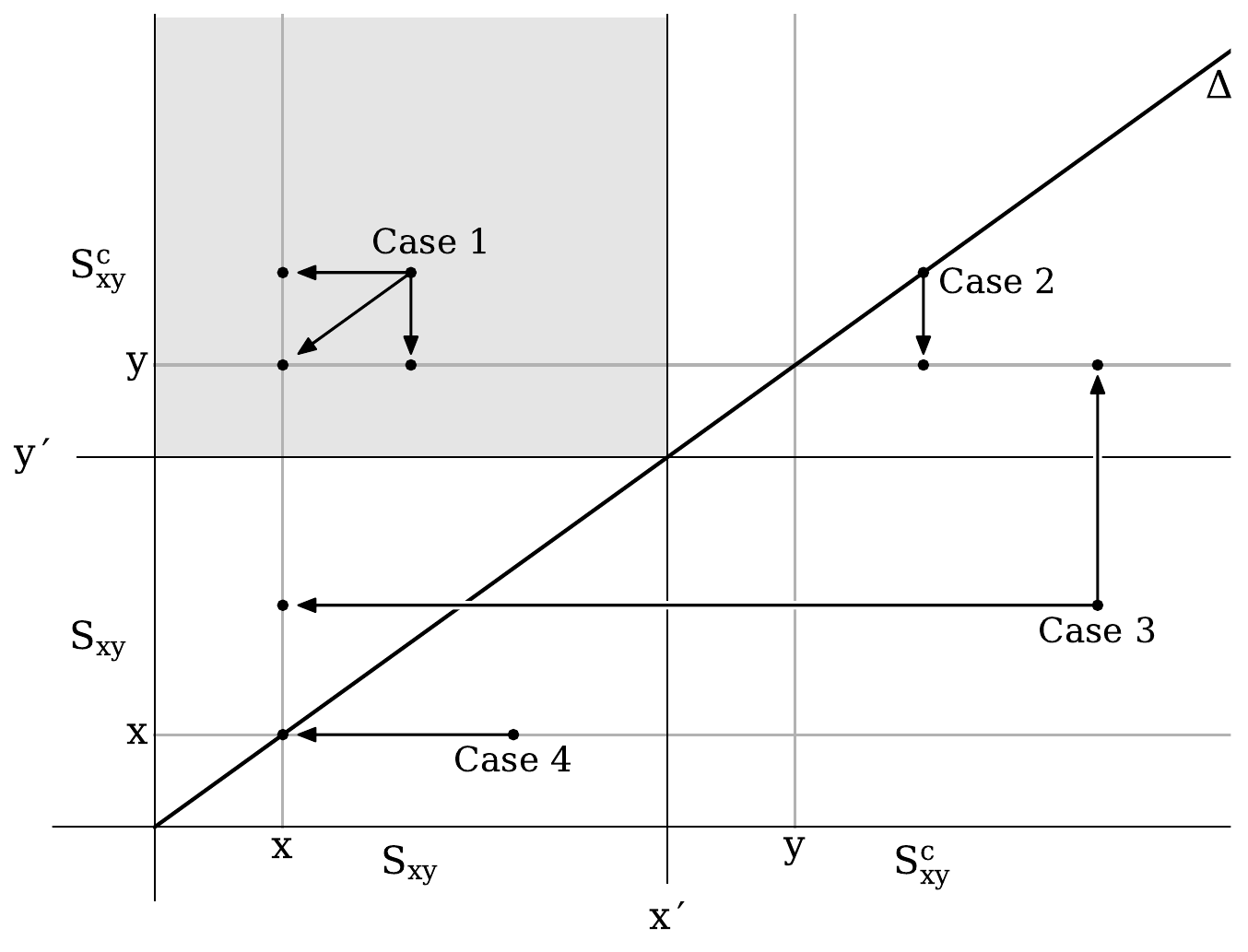}
\caption{
Diagram of acceptance scenarios considered in Theorem~\ref{thm:max}. The support of a maximal
coupling $\bp(\xy,\cd)$ is contained in the union of $\sxy \times \sxy^c$ (gray box) and $\Delta$
(the diagonal). Arrows illustrate the relationship of proposals $\xyp$ to transitions $(X,Y)$ under
different accept/reject combinations, with transitions outside the support of $\bp(\xy,\cd)$
forbidden almost surely. Case 1: the maximality of $\bp$ does not constrain the acceptance pattern
of proposals ${ \xyp \in \sxy \times \sxy^c} $. Case 2: proposals in $\Delta$ can be fully accepted
$(b_x=b_y=1)$ or fully rejected $(b_x=b_y=0)$, but $y'$ must be accepted if $x' \in S_{xy}^c$, and
$x'$ must be accepted if $y' \in S_{xy}$. Case 3: proposals in $(\sxy \times \sxy^c)^c \cap
\Delta^c$ must be fully rejected unless $y'=x$ or $x'=y$. Case 4: a proposal $\xyp$ outside the
support of $\bp(\xy, \cd)$ may be partially accepted $(b_x \neq b_y)$ if it yields a transition to
$(x,x)$ or $(y,y)$. \label{fig:plaid}
}
\end{figure}

Recall that by Corollary~\ref{cor:recog}, the maximality of a coupling $\bp \in \Gamma(P,P)$ is
equivalent to a condition on the support of each $\bp(\xy,\cd)$. Conditions 3-6 relate these support
constraints to the behavior of a proposal coupling $\bq$ and an acceptance indicator coupling $\bb$.
See Figure~\ref{fig:plaid} for an illustration of the acceptance scenarios considered in these
conditions and a visual intuition for why certain cases must be ruled out for $\bq$ and $\bb$ to
generate a maximal $\bp$.

\begin{proof}[Proof of Theorem~\ref{thm:max}]
Suppose $\bp \in \Gmax(P,P)$. By Theorem~\ref{thm:repr}, there exists a proposal coupling $\bq \in
\Gamma(Q,Q)$ and an acceptance indicator coupling $\bb$ such that for any $x, y \in \calX$ and $\bxy
\sim \bb(\xy,\xyp)$, we have $\P(b_x=1 \g x,y,x') = a(x,x')$ for $Q(x,\cdot)$-almost all $x'$, and
$\P(b_y=1 \g x,y,y')= a(y,y')$ for $Q(y,\cdot)$-almost all $y'$. Thus  Conditions 1 and 2 directly
follow from Theorem \ref{thm:repr}.

Since	$\bq$ and $\bb$ generate $\bp$, we have $\XY \sim \bp(\xy,\cd)$ where $X = b_x x' + (1-b_x) x$
and $Y=b_y y' + (1-b_y) y$. Since $\bp$ is maximal, Corollary~\ref{cor:recog} implies $0 = \bp\big(
\xy, ( S_{xy}^c \times \calX) \sm \Delta \big) = \P(X \in \sxy^c, X \neq Y \g x,y)$. Breaking this
up into the four possible acceptance scenarios $\bxy = (1,1)$, $(1,0)$, $(0,1)$, and $(0,0)$ yields
\eq{
	0 &= \E[ 1(x' \neq y', x' \in \sxy^c) \P(b_x=b_y=1 \g x,y,x',y') ] \hspace{23pt}
	0 = \E[ 1(x' \neq y, x'\in \sxy^c) \P(b_x=1,b_y=0 \g x,y,x',y') ] \\
	0 &= \E[ 1(y' \neq x, x \in \sxy^c) \P(b_x=0, b_y=1 \g x,y, x',y') ] \qquad
	0 = \E[ 1(x \in \sxy^c \sm \sy) \P(b_x=b_y=0 \g x,y,x',y') ].
}
In turn, these equations imply that for $\bq(\xy,\cd)$-almost all $\xyp$,
\eq{
	& \P(b_x=b_y=1 \g x,y,x',y') = 0 \text{ if } x' \neq y' \text{ and } x' \in \sxy^c \hspace{24pt}
	\P(b_x=1, b_y=0 \g x,y,x',y') = 0 \text{ if } x' \neq y \text{ and } x' \in \sxy^c \\
	& \P(b_x=0, b_y=1 \g x,y,x',y') = 0 \text{ if } y' \neq x \text{ and } x \in \sxy^c \qquad
	\P(b_x=b_y=0 \g x,y,x',y') = 0 \text{ if } x \neq y \text{ and } x \in \sxy^c.
}
This shows that the first either/or case of each of Conditions 3-6 are satisfied. Since $\bp$ is
maximal, Corollary~\ref{cor:recog} also implies $0 = \bp\big( \xy, (\calX \times S_{xy}) \sm \Delta
\big)$. Proceeding as above shows that the second either/or cases are also satisfied. So we conclude
that if $\bp \in \Gmax(P,P)$, then $\bp$ is generated by a $\bq$ and $\bb$ satisfying the six
conditions stated above.

For the converse, suppose that $\bq \in \Gamma(Q,Q)$ and a proposal coupling $\bb$ generate $\bp$
and satisfy the given hypotheses. Since Conditions 1 and 2 are equivalent to the conditions of
Theorem~\ref{thm:repr}, we have $\bp \in \Gamma(P,P)$. Now let $\xyp \sim \bq(\xy,\cd)$, $\bxy \sim
\bb(\xy,\xyp)$, $X = b_x x' + (1-b_x) x$, and $Y= b_y y' + (1-b_y) y$. $\bq$ and $\bb$ generate
$\bp$, so
\eq{
	& \bp(\xy, (S_{xy}^c \times \calX) \sm \Delta) = \P(X \in S_{xy}^c, X \neq Y \g x,y) \\
	& \  = \P(x' \in A, x' \neq y', b_x=1, b_y=1 \g x,y) + \P(x' \in A, x' \neq y, b_x=1, b_y=0 \g x,y) \\
	& \  + \P(y' \neq x, b_x=0, b_y=1 \g x,y) 1(x \in A) + \P(b_x=0, b_y=0 \g x,y) 1(x \in A) 1(x \neq y) = 0.
}
The last equality follows directly from Conditions 3-6 of the Theorem, with Condition 3 ensuring
that the first term equals zero, Condition 4 ensuring that the second term equals zero, and so on. A
similar argument yields ${\bp(\xy, (\calX \times S_{xy} ) \sm \Delta) = 0}$. By
Corollary~\ref{cor:recog}, $\bp$ is maximal if and only if there is a measurable set $S \in \scrF$
such that $\bp(\xy, (S^c \times \calX) \sm \Delta ) = \bp(\xy, (\calX \times S) \sm \Delta ) = 0$.
The argument above shows that $S_{xy}$ has these properties, so we conclude that $\bp \in
\Gmax(P,P)$.
\end{proof}

In Lemma~\ref{lem:maxtonon} and Appendix~\ref{ex:nonmax}, we observed that maximal proposal
couplings only weakly relate to maximal transition kernel couplings. Theorem~\ref{thm:max} shows
that any proposal coupling can give rise to a maximal coupling, as long as it satisfies the
conditions of the theorem for some acceptance indicator coupling $\bb$. In particular, such a $\bq$
must distribute probability over $\Delta$ so that $P(x, A) \wedge P(y, A) \leq \bq(\xy, A_\Delta)$
for all $A \in \scrF$.

As noted above, most of the MH couplings considered so far in the literature use a maximal proposal
coupling ${ \bq \in \Gmax(Q,Q) }$. However, the examples of~\citet{Jacob2020} and
\citet{o2021couplings} show that some of these perform better than others, and that the distribution
of meeting times often reflects the degree of contraction between chains when a meeting is not
proposed rather than the frequency with which meetings are proposed. These examples suggest that
non-maximal proposal couplings cannot be dismissed out of hand. The results above sharpen this
conclusion by observing that with the right acceptance indicator coupling, a non-maximal proposal
coupling can still produce transition kernel meetings at the optimal rate.

%%%%%%%%%%%%%%%%%%%%%%%%%%%%%%%%%

\section{Two-step representations of common MH algorithms}
\label{sec:appn}

In Theorem~\ref{thm:repr} and Corollary~\ref{cor:repr_non_atomic}, we showed that any coupling of
MH-like transition kernels can be expressed as a proposal coupling followed by an acceptance
indicator coupling. In Sections~\ref{sec:discrete} and \ref{sec:continuousmh}, we apply these
results to couplings of the MH algorithms on a discrete state space, the random walk Metropolis
algorithm, and the Metropolis-adjusted Langevin algorithm. In Section~\ref{sec:minorization} we
consider a two-step representation of the transition kernel coupling used in Nummelin splitting
\citep{rosenthal1995minorization, jones2001honest}. An overall impression from these examples is
that our results are simple to apply and can sometimes recast complex kernel couplings in
a more approachable, two-step form.

\subsection{The Metropolis--Hastings algorithm on a finite state space}
\label{sec:discrete}

We begin by considering a finite state space $\calX$ with an MH-like transition kernel $P$ based on
a proposal kernel $Q$ with ${Q(x, \sx) = 0}$ for all $x \in \calX$. Fix any transition kernel
coupling ${\bp \in \Gamma(P, P)}$. We write $r(x) = P(x,\sx)$ for the probability of a transition
from $x$ to itself, and in this subsection we also write $\bar \Theta(\xy, \xyp)$ in place of ${\bar
	\Theta(\xy, \{x'\} \times \{y' \})}$, where $\bar \Theta$ is any joint kernel in $\Gamma(Q,Q)$ or
$\Gamma(P,P)$.

Corollary~\ref{cor:repr_non_atomic} shows that $\bp$ can be generated by a particular proposal
coupling $\bq$ and acceptance indicator coupling $\bb$. Condition 1 of that result implies that
$\bq_\perp(\xy, \cdot) = \bp_\perp(\xy, \cdot) = 0$, since the base measure on $\calX \times \calX$
is counting measure. Condition 2 then shows that we can express $\bq(\xy, \xyp)$ as the sum of four
terms, $\Phi_{11}, \Phi_{10}, \Phi_{01}, \sas \Phi_{00}$. These all equal zero if $x' = x$ or $y' =
y$, since then $0 \leq \Phi_{ij}(\xy, \xyp) \leq \bq(\xy, \xyp) = 0$ by our `non-laziness' condition
on $Q$. For $x' \neq x$ and $y' \neq y$, we have
\ec{
	\Phi_{11}(\xy, \xyp) = \bp(\xy, \xyp) \\
	\Phi_{10}(\xy, \xyp) = \bp(\xy, (x', y)) \mu(y, y') 1(r(y) > 0) \\
	\Phi_{01}(\xy, \xyp)  =  \mu(x, x') 1(r(x) > 0) \bp(\xy, (x, y')) \\
	\Phi_{00}(\xy, \xyp) =  \mu(x, x') \mu(y, y') 1(r(x) > 0) 1(r(y) > 0) \bp(\xy, \xy).
}
If $x' \neq x$ and $r(x) > 0$, then $\mu(x, x') = (Q(x, \{x'\}) - P(x, \{x'\}))/r(x)$.
From here, it is easy to confirm that ${\bq \in \Gamma(Q,Q)}$.
Summing $\bq(\xy, \xyp)$ over $y'$ yields $Q(x, \{x'\})$, since
\eq{
	& \bp(\xy, \{x'\} \times \sy^c) + \bp(\xy, (x', y)) 1(r(y) > 0) \\
	& \quad + \mu(x, x') 1(r(x) > 0) r(x) + \mu(x, x') 1(r(x) > 0) 1(r(y) > 0) \bp(\xy, \xy) \\
	& = P(x, \{ x' \}) + (Q(x, \{x'\}) - P(x, \{x'\})) = Q(x, \{x'\}).
}
A similar argument shows that $\sum_{x'} \bq(\xy, \xyp) = Q(y, \{y'\})$.

We can also derive the necessary acceptance indicator coupling by using the $\Phi$ formulas above
with the proof method of Lemma~\ref{lem:rnderivs}. Let $\xy$ be the current state and let $\xyp$ be
a proposed state pair with $\bq(\xy, \xyp) >0$. Then the probability of accepting both proposals
will be $\Phi_{11}(\xy,\xyp) / \bq(\xy, \xyp)$, the probability of accepting $x'$ and rejecting $y'$
will be $\Phi_{10}(\xy,\xyp) / \bq(\xy, \xyp)$, and so on. These acceptance probabilities are
defined in terms of Radon--Nikodym derivatives, which are simple to compute in this case due to the
discrete state space.

The current example represents a significant specialization of the hypotheses of
Theorem~\ref{thm:repr} and Corollary~\ref{cor:repr_non_atomic}. Still, it provides a useful
intuition for how our methods work in more general situations. In this case, we see that
$\bq(\xy, \xyp)$ arises as a simple combination of the values of $\bp(\xy, \cdot)$ under the four
possible acceptance scenarios, using weights derived from $\mu$. We recall from
Lemma~\ref{lem:qualitative} that $\mu(x,x') = \P(x' = x \g b_x = 0, x)$ under suitable conditions.
Thus we see that the form of $\bq$ in our existence results is the one that assumes the least
possible dependence between $x'$ and $y'$ when one or both of these are rejected. This observation
holds in the general case, as well.

\subsection{The random walk Metropolis and Metropolis-adjusted Langevin algorithms}
\label{sec:continuousmh}

Next, we consider two popular MH kernels on the continuous state space $\calX = \R^n$. We write
$\lambda$ for Lebesgue measure on $\calX$, and we take $\pi$ to be a target distribution on $\calX$ with
density $\pi(\cdot)$ with respect to $\lambda$. We will also write $Q$ for any proposal kernel
on $\calX$ with density $q(x, \cdot)$, $a$ for the associated MH acceptance rate, $P$ for the
resulting transition kernel, and $\bp \in \Gamma(P,P)$ for some fixed kernel coupling based on $P$.

For the Random Walk Metropolis (RWM) algorithm, we use a proposal density $q$ such that $q(x, x') =
g(\lVert x - x' \rVert)$ for some ${g : \R^+ \to \R^+}$. We also consider the Metropolis-adjusted
Langevin algorithm (MALA)~\citep{roberts1996exponential}, which uses a discretized approximation to
the Langevin diffusion $\diff X_t = \frac 12 \nabla \log\pi(X_t) \diff t + \diff W_t$, where $W_t$
represents a Brownian motion. For MALA, we use the proposal distribution $Q(x,\cdot) = \N(x+\tau
\nabla \pi(x), 2\tau \mathbb I)$, where $\tau > 0$ serves as a tuning parameter. This distribution
has density $q(x,y) \propto \exp(-\frac{1}{4\tau}\lVert y - x - \tau \nabla \log\pi(x) \rVert^2)$.
MALA's use of gradient information makes it more computationally expensive, but it tends to
outperform RWM in high dimensions
\citep{roberts1998optimal}.

For both algorithms, Corollary \ref{cor:repr_non_atomic} implies the existence of a proposal
coupling $\bq \in \Gamma(Q,Q)$ that can be written as ${\bq = \bq_\perp + \bq_\ll}$, where $\bq_\perp$
is singular and $\bq_\ll$ is absolutely continuous with respect to the base measure $\lambda \times
\lambda$. Condition~1 of Corollary \ref{cor:repr_non_atomic} tells us that $\bq_\perp(\xy, \cdot) =
\bp_\perp(\xy, \cdot \cap (\sx^c \times \sy^c))$, the singular part of the transition kernel
coupling after removing the probability at the current state pair $(x,y)$. Condition 2 of that result tells
us that the continuous part $\bq_\ll(\xy, \cdot)$ of the proposal coupling will have density
\eq{
	\bar q(\xy, \xyp) = p(\xy, \xyp) + \bar p_y(x, x') m(y,y') + m(x,x') \bar p_x(y, y') + m(x,x') m(y,y') \bar r(x,y).
}
Here $\bar p_x, \bar p_y, m, \sas \bar r$ are all as defined in Section \ref{sec:cor}.

By Condition 3, proposals $\xyp$ in the support of $\bq_\perp(\xy, \cdot)$ are automatically
accepted. Otherwise, we can determine the probabilities of accepting one, both, or neither of the
proposals by dividing the relevant term of the density above by $\bar q$. Thus the probability of
accepting both $x'$ and $y'$ will be $p(\xy, \xyp)/\bar q(\xy,\xyp)$, the probability of accepting
$x'$ and rejecting $y'$ will be $\bar p_y(\xy x') m(y,y') / \bar q(\xy, \xyp)$, and so on.

We conclude that the differences between the two-step representation of couplings of these
algorithms are contained entirely within the functions that make up the expression for $\bar q$. A
similar conclusion also holds for more complicated algorithms such as Hamiltonian Monte
Carlo~\citep{Duane:1987, Neal1993, neal2011mcmc} and Metropolis-within-Gibbs~\citep{Metropolis1953,
	gelfand1990sampling}, for which the functions $\bar p$ and $m$ require a correspondingly greater
effort to derive.

\subsection{Coupling with a minorization condition} \label{sec:minorization}

Let $P$ be an MH-like transition kernel generated by a proposal kernel $Q$ and an acceptance rate
function $a$, and assume that $P$ has stationary distribution $\pi$. As above, we assume that $Q(x,
\cdot) \ll \lambda$ and $Q(x, \sx) = 0$ for all $x \in \calX$. The kernel $P$ satisfies a
minorization condition corresponding to a set $\calC \subseteq \calX$ if there exists an $\eps > 0$ and
a probability measure $\nu$ on $(\calX, \scrF)$ such that $P(x,A) \geq  \epsilon\, \nu(A)$ for all
$x \in \calC$ and all $A \in \scrF$. If $\calC = \calX$, then a simple argument shows that $\lVert
P^n(x,\cdot) - \pi(\cdot) \rVert_{\TV} \leq M (1 - \epsilon)^t$ for some $M>0$ and each positive
integer $t$~\citep{meyn2012markov}. If $\calC$ is a proper subset of $\calX$, further assumptions
are needed to control the distribution of return times of the chain to $\calC$ and so to obtain
convergence rate results.

Minorization arguments often use a coupling technique called Nummelin splitting
\citep{nummelin1978uniform, athreya1978new, jones2001honest, gelman2010handbook} to bound the
convergence rate of such a chain to its stationary distribution. Suppose that we run two identically
distributed Markov chains $(X_t)$ and $(Y_t)$, with one initialized arbitrarily and the
other initialized from its stationary distribution $\pi$. If both chains occupy $\calC$ at the same iteration
$t$, then we flip a coin with probability $\eps$ of heads. On heads, we draw
$X_{t+1} = Y_{t+1} \sim \nu$, and otherwise we draw $X_{t+1}$ and $Y_{t+1}$ independently from their
residual distributions so that $X_{n+1} \sim \nu^r_x := (P(X_n,\cdot) - \epsilon \nu(\cdot) / (
1-\epsilon)$ and $Y_{n+1}\sim \nu^r_y := (P(Y_n,\cdot) - \epsilon \nu(\cdot))/(1-\epsilon)$.
Once the chains meet, we update both chains using the same transitions from $P$, so that they
remain together at all subsequent iterations. Otherwise, we repeat this process whenever both chains
occupy $\calC$, until meeting occurs.

The procedure above yields a coupling $\bp \in \Gamma(P,P)$, defined without direct
reference to the proposal or acceptance distributions. Alternatively, we can use Theorem
\ref{thm:repr} to represent this coupling in terms of a proposal coupling $\bq$ and an acceptance
indicator coupling $\bb$. For every $x,y \in \calC$, the kernel coupling $\bp(\xy, \cdot)$ described
above is a mixture of a measure on the diagonal $\Delta = \{(x,x) : x \in \calX) \}$ and a measure
on its complement in $\calX \times \calX$. In particular, for any measurable set $S \in
\scrF \otimes \scrF$ we have $\bp(\xy, S) =\eps\, \nu(S \cap \Delta) + (1 - \eps) (\nu^r_x \times
\nu^r_y) (S \sm \Delta)$.

The key step in finding a two-step representation of $\bp$ is to identify an appropriate coupled
acceptance mechanism $\Phi = (\Phi_{11}, \Phi_{10}, \Phi_{01}, \Phi_{00})$.
Due to our assumptions on $Q$, we have
$\beta(x) = 0$ if $r(x)> 0$ and $1$ otherwise. We also have
\eq{
\mu(x,A_x) & = \begin{cases}
	\frac{Q(x,A_x) - P(x,A_x\setminus\{x\})}{r(x)} & \text{if } r(x) > 0 \\
	1(x \in A_x) & \text{otherwise}.
\end{cases}
}
A similar description holds for $\mu(y, A_y)$. From the proof of Lemma \ref{lem:camdpc}, we can
write the components of $\Phi$ as follows:
\ec{
\Phi_{11}(\xy, A_x\times A_y)= \bp \big( \xy, A_x\times A_y \cap (\sx^c \times \sy^c) \big)\\
\Phi_{10}(\xy, A_x\times A_y) = \bp( \xy, (A_x \cap \sx^c) \times \sy) \times \mu(y, A_y)\\
\Phi_{01}(\xy, A_x\times A_y) = \mu(x,A_x)\times\bp( \xy,   \sx \times (A_y \cap \sy^c)) \\
\Phi_{00}(\xy, A_x\times A_y) =  r(x,y) \mu(x,A_x)\times\mu(y,A_y).
}
Here $r(x,y) =\bar P((x,y),\{x\}\times\{y\})$ is the probability that the joint chain stays put at
$(x,y)$. In this example the two chains evolve independently independent when meeting does not
occur, and so we have $r(x,y) = r(x) r(y)$. With $\Phi$ defined as above, we take
${\bar Q =\Phi_{00} +\Phi_{01} + \Phi_{10}+\Phi_{11}}$ to be our proposal coupling.

Next we consider the acceptance indicator coupling $\bb$.
Lemma ~\ref{lem:rnderivs} implies the existence of acceptance indicators $\bxy \sim \bb(\xy,\xyp)$ such that
$\P (b_x = i, b_y = j \g x,y,x',y') = \phi_{ij}\big( \xy, \xyp \big)$
where ${\phi_{ij}(\xy,\cdot)} = {\diff \Phi_{ij}(\xy, \cdot)/\diff Q(\xy, \cdot)}$.
If we write each term of $\Phi_{ij}$ in its integral form, then $\phi_{ij}$ can also be
evaluated at every point  as a ratio of densities. For example, let $x' \neq x$, $y' \neq y$ be two
different points in $\calX$. Then the probability density of $\bq$ for a move from $\xy$ to $\xyp$
will be
\eq{
q(\xy, \xyp)
& = (p(x, x') -\epsilon \nu(x'))(p(y,y') - \epsilon \nu(y'))
+ (p(x,x') - \epsilon \nu(x')) (q(y,y') - p(y,y')) \\
& \quad + (p(x,x') -  q(x,x')) (q(y,y') - \epsilon \nu(y')) +
(q(x,x') -p(x,x'))  (q(y,y') - p(y,y')).
}
Here $p(x, x') = q(x, x') a(x,x')$ is the transition kernel density from $x$ to $x'$. Thus, when $x'
\neq x$, $y' \neq y$, and $x' \neq y'$, we have the following expressions for the distribution of
the acceptance indicator pair $\bxy$ conditional on the current state $\xy$ and the proposal state
$\xyp$:
\eq{
p_{11} &:= \P(b_x = 1, b_y = 1 \g x,y,x',y')
= \frac{ ( p(x,x') -\epsilon \nu(x'))(p(y,y') - \epsilon \nu(y')) }{ q(\xy, \xyp) }\\
p_{10} &:= \P(b_x = 1, b_y = 0 \g x,y,x',y')
= \frac{ ( p(x,x') - \epsilon \nu(x')) (q(y,y') - p(y,y')) }{ q(\xy, \xyp) }\\
p_{01} &:= \P(b_x = 0, b_y = 1 \g x,y,x',y')
= \frac{ ( p(x,x') -  q(x,x')) (q(y,y') - \epsilon \nu(y')) }{ q(\xy, \xyp)} \\
p_{00} &:= \P(b_x = 0, b_y = 0 \g x,y,x',y')
= \frac{ ( q(x,x') -p(x,x')) (q(y,y') - p(y,y'))  }{ q(\xy, \xyp)}.
}
The marginal probability of accepting the `$x$' move is $p_{10} + p_{11}$ and the probability of
accepting the  `$y$' move is $p_{01} + p_{11}$, consistent with Lemma \ref{lem:rndproperties}.  As a
sanity check,  the probability density of moving from $\xy$ to $\xyp$ according to our coupling and
acceptance mechanism is $q(\xy,\xyp)\times p_{11} = (p(x,x') -\epsilon \nu(x')) \cdot(p(y,y') -
\epsilon \nu(y'))$, which is indeed the probability density of the minorization coupling $\bar
P(\xy, \cdot)$ at point $\xyp$. The cases with $y' = x'$, $x' = x$, and $y' = y$ can be calculated
in the same way.

%%%%%%%%%%%%%%%%%%%%

\section{Discussion}
\label{sec:discussion}

Couplings are a topic of great theoretical and practical interest. In theoretical work, couplings
provide a powerful tool for proving results on Markov chains and other stochastic processes. They
are also related to assignment problems under marginal constraints, a topic of interest since at
least~\citet{birkhoff1946three}. In applications, couplings provide a foundation for techniques
related to convergence diagnosis, variance reduction, and unbiased estimation, as described in
Section~\ref{sec:intro}.

In this paper we considered kernel and maximal kernel couplings of the MH algorithm and related
methods. In Theorem~\ref{thm:repr}, we showed that any transition kernel coupling can be expressed
in terms of a proposal coupling and an acceptance indicator coupling with certain properties.
Conversely, any joint kernel generated this way must be a legitimate transition kernel coupling. In
Theorem~\ref{thm:max} we took this a step further, showing that maximal kernel couplings correspond
to pairs of proposal and acceptance couplings subject to a short list of additional conditions.
These results provide a unified approach for describing kernel couplings of MH-like algorithms. In
principle, our results could be extended to other algorithms, such as multiple-try MH
\citep{liu2000multiple}, the guided walk of~\citet{gustafson1998guided}, and other irreversible MCMC
methods. However, it may be more difficult to derive simple two-step expressions like those
described in Sections~\ref{sec:cor} and \ref{sec:appn} for these other algorithms.

Throughout this study we have viewed the transition kernel $P$ on a state space $(\calX, \scrF)$ as
arising from a fixed proposal kernel $Q$ and acceptance rate function $a$. Alternatively, we could
have started with any kernel $P$ that is weakly dominated by any other kernel $Q$, in the sense that
$P(x, A) \leq Q(x, A)$ for all $x \in \calX$ and all $A \in \scrF$ with $x \not \in A$. This is
essentially the ordering condition of~\citet{peskun1973optimum}, as generalized to infinite state
spaces by~\citet{tierney1998note}. In light of the above, our Theorem~\ref{thm:repr} can be
interpreted as stating that given a kernel $P$ weakly dominated by a kernel $Q$, and given a
Radon--Nikodym derivative $a(x, \cdot) = \diff P(x, \cdot) / \diff Q(x, \cdot)$ away from the
current state $x$, then every kernel coupling ${\bp \in \Gamma(P, P)}$ can be generated by a
coupling $\bq \in \Gamma(Q, Q)$ followed by an acceptance indicator coupling $\bb$ based on $a$.
This interpretation turns our result into a joint probability existence and characterization theorem
that includes MH-like transition kernel couplings as a special case.

Looking forward, we believe that our results can be used to design new and more efficient couplings
for the MH algorithm. We also expect them to support theoretical work on kernel couplings for this
broad class of discrete-time Markov chains. It may be possible to derive meeting time bounds or
information on the spectrum of $\bp$ based on data about $\bq$ and $\bb$, using a version of  the
drift and minorization arguments of~\citet{rosenthal1995minorization, rosenthal2002quantitative} or
the techniques of~\citet{atchade2007geometric}. A deeper understanding of kernel couplings may also
support optimization arguments like those of~\citet{boyd2004fastest, boyd2006fastest,
	boyd2009fastest}. Finally, in analogy with our description of maximal couplings in
Theorem~\ref{thm:max}, it may be possible to characterize the optimal transport couplings in
$\Gamma(P,P)$ in terms of the properties of proposal and acceptance indicator couplings.

This study represents a step toward understanding the properties of $\Gamma(P,P)$ and $\Gmax(P,P)$
in terms of those of $\Gamma(Q,Q)$ and the associated acceptance indicator couplings. Markovian
couplings which are efficient in the sense of~\citet{aldous1983random} are well understood only in
special cases. Beyond that, few other bounds on the efficiency of Markovian couplings are
well-established~\citep{burdzy2000efficient}. There remain much to learn about such couplings for
MH-like chains. A clearer understanding of the set of Markovian couplings will make it possible to
identify better options for use in practice and to determine how efficient couplings can be for this
important class of algorithms.

%%%%%%%%%%%%%%

\begin{appendix}

\section{Simple examples}
\label{sec:example}

In this appendix we present two simple examples to illustrate concepts from the main text. In the
first example, we explore the definitions and proof methods of Section~\ref{sec:kercoup} using an MH
kernel on a two-point state space. In the second example, we consider a three-point state space to
show that some maximal kernel couplings can only be generated from non-maximal proposal couplings,
as discussed in Section~\ref{sec:maximal}.

\subsection{Two-step representation of a transition kernel coupling}
\label{ex:running}

Let $\calX = \{1,2\}$, $\scrF =
2^\calX$, and assume a current state pair $\xy = (1,2)$. For a small finite state space like this, it is
convenient to represent a kernel $\Theta : \calX \times \scrF \to [0,1]$ by
$\Theta(z, \cd) = \big( \Theta(z,\{1\}), \Theta(z,\{2\}) \big)$
and to represent a coupling ${\bar \Theta: (\calXp) \times (\scrFp) \to [0,1]}$
using a table like the following:
\eq{
	\raisebox{3.5pt}{ $ \bar \Theta(\xy, \cd)  = $ } \hspace{-5pt}
	\begin{blockarray}{rccl}
		& {\stackrel{x}{\scriptstyle \downarrow}} & &  \\
		\begin{block}{ r |c|c| l }
			\cline{2-3}
			{\scriptstyle 1}  & \Theta(\xy, (1,1) ) & \Theta(\xy, (2,1)) & \bigstrut \\
			\cline{2-3}
			{\scriptstyle 2} & \Theta(\xy, (1,2) ) & \Theta(\xy, (2,2)) & {\hspace{-4pt} \scriptstyle \leftarrow y} \bigstrut \\
			\cline{2-3}
		\end{block}
		\  & {\scriptstyle 1} & {\scriptstyle 2}
	\end{blockarray}
}\\[-1.75em]
The columns of this table represent the destination values for the $x$ chain (e.g. the values of $x'$ or $X$)
while the rows represent the equivalent for the $y$ chain. The small `$x$' and `$y$' serve
as a reminder of the current state pair.

We assume a uniform proposal distribution and an MH acceptance rate function $a$ based on a target
distribution ${\pi = (\sfrac{1}{3}, \sfrac{2}{3})}$, so that $a(1, \cd) = (1,1)$ and $a(2,\cd) = (\h,
1)$. Together these imply a transition kernel $P$ with $P(1,\cd) = (\h,\h)$ and $P(2,\cd) =
(\sfrac{1}{4}, \sfrac{3}{4})$. Simple algebra shows that any proposal and transition coupling based
on the kernels above must take following forms, for some $\rho \in [0,\h]$ and $\lambda \in [0,
\sfrac{1}{4}]$:
\eq{
	\raisebox{3.5pt}{ $ \bq((1,2), \cd) = $ } \hspace{-5pt}
	\begin{blockarray}{rccl}
		& {\stackrel{x}{\scriptstyle \downarrow}} & &  \\
		\begin{block}{ r |c|c| l }
			\cline{2-3}
			{\scriptstyle 1}  & \rho & \h - \rho & \bigstrut \\
			\cline{2-3}
			{\scriptstyle 2} &\h-\rho & \rho & {\hspace{-4pt} \scriptstyle \leftarrow y} \bigstrut \\
			\cline{2-3}
		\end{block}
		\  & {\scriptstyle 1} & {\scriptstyle 2}
	\end{blockarray}
	\hspace{1em}
	\raisebox{3pt}{ $ \bp((1,2), \cd) = $ } \hspace{-5pt}
	\begin{blockarray}{rccl}
		& {\stackrel{x}{\scriptstyle \downarrow}} & &  \\
		\begin{block}{ r |c|c| l }
			\cline{2-3}
			{\scriptstyle 1}  & \lambda & \sfrac{1}{4} -\lambda & \bigstrut \\
			\cline{2-3}
			{\scriptstyle 2} &\h - \lambda & \lambda + \sfrac{1}{4} & {\hspace{-4pt} \scriptstyle \leftarrow y} \bigstrut \\
			\cline{2-3}
		\end{block}
		\  & {\scriptstyle 1} & {\scriptstyle 2}
	\end{blockarray}
}\\[-1em]
Definition~\ref{def:cam}, on coupled acceptance mechanisms, raises two questions: whether its
Condition 3 is independent of its Conditions 1 and 2, and whether a given $\bp \in \Gamma(P,P)$ can
be related to $\bq \in \Gamma(Q,Q)$ by more than one coupled acceptance mechanism. We answer both
questions in the affirmative.

Suppose that we have the following proposal and transition kernel
couplings, corresponding to the case of $\rho = \sfrac{1}{4}$ and $\lambda = 0$ above:
\eq{ \raisebox{3pt}{ $ \bq((1,2), \cd)
		= $ } \hspace{-5pt}
	\begin{blockarray}{rccl}
		& {\stackrel{x}{\scriptstyle \downarrow}} & &  \\
		\begin{block}{ r |c|c| l }
			\cline{2-3}
			{\scriptstyle 1}  & \sfrac{1}{4} & \sfrac{1}{4} & \bigstrut \\
			\cline{2-3}
			{\scriptstyle 2} &\sfrac{1}{4} & \sfrac{1}{4} & {\hspace{-4pt} \scriptstyle \leftarrow y}  \bigstrut  \\
			\cline{2-3}
		\end{block}
		\  & {\scriptstyle 1} & {\scriptstyle 2}
	\end{blockarray}
	\hspace{2.5em}
	\raisebox{3pt}{ $ \bp((1,2), \cd) = $ } \hspace{-5pt}
	\begin{blockarray}{rccl}
		& {\stackrel{x}{\scriptstyle \downarrow}} & &  \\
		\begin{block}{ r |c|c| l }
			\cline{2-3}
			{\scriptstyle 1}  & 0 & \sfrac{1}{4} &  \bigstrut  \\
			\cline{2-3}
			{\scriptstyle 2} &\h & \sfrac{1}{4} & {\hspace{-4pt} \scriptstyle \leftarrow y}  \bigstrut  \\
			\cline{2-3}
		\end{block}
		\  & {\scriptstyle 1} & {\scriptstyle 2}
	\end{blockarray}
}
By Lemma~\ref{lem:cases}, we observe that the first two conditions of Definition~\ref{def:cam}
require any coupled acceptance mechanism $\Phi$ relating $\bq$ and $\bp$ to take the following form
for real numbers $a, b, c, d, \sas e$:
\begin{alignat*}{2}
	& \raisebox{3pt}{ $ \Phi_{11}((1,2), \cd) = $ } \hspace{-5pt}
	\begin{blockarray}{rcc} \\
		\begin{block}{ r |c|c| }
			\cline{2-3}
			{\scriptstyle 1}  &  0 & \sfrac{1}{4} \bigstrut \\
			\cline{2-3}
			{\scriptstyle 2} & a & b \bigstrut  \\
			\cline{2-3}
		\end{block}
		\  & {\scriptstyle 1} & {\scriptstyle 2}
	\end{blockarray}
	\qquad
	&& \raisebox{3pt}{ $ \Phi_{10}((1,2), \cd) = $ } \hspace{-5pt}
	\begin{blockarray}{rcc} \\
		\begin{block}{ r |c|c| }
			\cline{2-3}
			{\scriptstyle 1}  & d & 0 \bigstrut  \\
			\cline{2-3}
			{\scriptstyle 2} & c & 1/4	- b \bigstrut  \\
			\cline{2-3}
		\end{block}
		\  & {\scriptstyle 1} & {\scriptstyle 2}
	\end{blockarray} \\
	& \raisebox{3pt}{ $ \Phi_{01}((1,2), \cd) = $ } \hspace{-5pt}
	\begin{blockarray}{rcc} \\
		\begin{block}{ r |c|c| }
			\cline{2-3}
			{\scriptstyle 1}  & 0 & 0 \bigstrut  \\
			\cline{2-3}
			{\scriptstyle 2} & e & 0 \bigstrut  \\
			\cline{2-3}
		\end{block}
		\  & {\scriptstyle 1} & {\scriptstyle 2}
	\end{blockarray}
	&& \raisebox{3pt}{ $ \Phi_{00}((1,2), \cd) = $ } \hspace{-5pt}
	\begin{blockarray}{rcc} \\
		\begin{block}{ r |c|c| }
			\cline{2-3}
			{\scriptstyle 1}  & \sfrac{1}{4}-d & 0 \bigstrut  \\
			\cline{2-3}
			{\scriptstyle 2} & \sfrac{1}{4}-a-c-e & 0 \bigstrut  \\
			\cline{2-3}
		\end{block}
		\  & {\scriptstyle 1} & {\scriptstyle 2}
	\end{blockarray}
	\\	\end{alignat*}\\[-3em]
Thus $(\Phi_{11} + \Phi_{10})(\xy, \sx \times \calX) = a + c + d$, and  $(\Phi_{11} +
\Phi_{01})(\xy, \calX \times \sy) = a + b + e$. Observe that ${a + c + d} = Q(x, \sx) = \h$ and $a +
b + e = Q(y, \sy) = \h$ do not automatically hold without imposing the third condition of
Definition~\ref{def:cam}. Therefore, this condition is not implied by the other two. Also note that
even with that condition, any choice of $a,b,c$ with $0 \leq c \leq b \leq \sfrac{1}{4}$ and $0 \leq
a \leq \sfrac{1}{4} - b \vee c$ will produce a valid $\Phi$ relating $\bq$ and $\bp$. Thus we see
that more than one coupled acceptance mechanism can relate the same proposal coupling $\bq$ and
transition kernel coupling $\bp$. Consequently, there can be more than one
choice of acceptance indicator coupling $\bb$ such that $\bq$ and $\bb$ generate $\bp$.

We continue with our two-state example to demonstrate the construction of a coupled acceptance
mechanism $\Phi$ and proposal coupling $\bq$ as in the proof of Lemma~\ref{lem:camdpc}. Under the
assumptions above, $\alpha_0$, $\alpha_1$, $\beta$, $\mu$ will take the following values:
\begin{alignat*}{3}
	\alpha_0(1,\cd) & = (0,0) \qquad
	\hspace{1ex} \alpha_1(1,\cd) && = (\h,\h) \qquad
	\mu(1,\cd) && = (1,0) \\
	\alpha_0(2,\cd) & = (\sfrac{1}{4},0) \qquad
	\alpha_1(2,\cd) && = (\sfrac{1}{4},\h) \qquad
	\mu(2,\cd) &&= (1,0) \\
	& \hspace{16ex} \beta(\cd) && = (1, \sfrac{2}{3}).
\end{alignat*}
We continue to use a transition kernel coupling $\bp$ with the following values:
\eq{
	\raisebox{3pt}{ $ \bp((1,2), \cd) = $ } \hspace{-5pt}
	\begin{blockarray}{rccl}
		& {\stackrel{x}{\scriptstyle \downarrow}} & &  \\
		\begin{block}{ r |c|c| l }
			\cline{2-3}
			{\scriptstyle 1}  & 0 & \sfrac{1}{4} & \bigstrut \\
			\cline{2-3}
			{\scriptstyle 2} &\h & \sfrac{1}{4} & {\hspace{-4pt} \scriptstyle \leftarrow y} \bigstrut  \\
			\cline{2-3}
		\end{block}
		\  & {\scriptstyle 1} & {\scriptstyle 2}
	\end{blockarray}
}\\[-1.5em]
Following the construction in the proof of Lemma \ref{lem:camdpc} and the values of $\alpha_0,
\alpha_1, \mu, \sas \beta$ shown above, we have
\eq{
	\raisebox{10pt}{ $ \Phi_{11}((1,2), \cd) = $ } \hspace{-5pt}
	\begin{blockarray}{rcc}
		\begin{block}{ r |c|c|}
			\cline{2-3}
			{\scriptstyle 1}  & 0 & \sfrac{1}{4}  \bigstrut  \\
			\cline{2-3}
			{\scriptstyle 2} &\sfrac{1}{3} & \sfrac{1}{6} \bigstrut  \\
			\cline{2-3}
		\end{block}
		\  & {\scriptstyle 1} & {\scriptstyle 2}
	\end{blockarray}
	\qquad
	\raisebox{10pt}{ $ \Phi_{10}((1,2), \cd) = $ } \hspace{-5pt}
	\begin{blockarray}{rcc}
		\begin{block}{ r |c|c| }
			\cline{2-3}
			{\scriptstyle 1}  & \sfrac {1}{6} & \sfrac{1}{12} \bigstrut  \\
			\cline{2-3}
			{\scriptstyle 2} &0& 0 \bigstrut  \\
			\cline{2-3}
		\end{block}
		\  & {\scriptstyle 1} & {\scriptstyle 2}
	\end{blockarray}
}\\[-1em]
Similar calculations show that both $\Phi_{01}((1,2), \cd)$ and $\Phi_{00}((1,2),\cd)$ consist entirely of zeros.
Finally, since ${\bq(\xy, \cd)} = {(\Phi_{11}+\Phi_{10}+\Phi_{01}+\Phi_{00})(\xy,\cd)}$, we have
\eq{
	\raisebox{3pt}{ $ \bq((1,2), \cd) = $ } \hspace{-5pt}
	\begin{blockarray}{rccl}
		& {\stackrel{x}{\scriptstyle \downarrow}} & &  \\
		\begin{block}{ r |c|c| l }
			\cline{2-3}
			{\scriptstyle 1}  & \sfrac{1}{6} & \sfrac{1}{3} & \bigstrut  \\
			\cline{2-3}
			{\scriptstyle 2} &\sfrac{1}{3}& \sfrac{1}{6}& {\hspace{-4pt} \scriptstyle \leftarrow y} \bigstrut  \\
			\cline{2-3}
		\end{block}
		\  & {\scriptstyle 1} & {\scriptstyle 2}
	\end{blockarray}
}\\[-1.5em]
We see that this proposal coupling has the marginal distributions $Q(1, \cd) = Q(2, \cd) = (\h,\h)$, as desired.

Finally, we derive an acceptance indicator coupling $\bb$ such that it and $\bq$ will generate the given $\bp$.
We proceed according to the construction in the proof of Lemma~\ref{lem:rnderivs}. There,
we defined the acceptance indicator coupling so that if ${\bxy \sim \bb(\xy, \xyp)}$, then for
$i,j \in \{0, 1\}$,
\eq{
	\P(b_x=i, b_y=j \g x, y, x', y') = \phi_{ij}(\xy, \xyp) = \frac{ \diff \Phi_{ij}(\xy, \cdot)}{\diff \bq(\xy, \cdot)}(x', y').
}
Due to the discrete state space, we can evaluate this Radon--Nikodym derivative by dividing the values
of $\Phi_{ij}(\xy, \xyp)$ by those of $\bq(\xy, \xyp)$. This yields the following joint acceptance
probabilities:
\begin{alignat*}{2}
	\raisebox{10pt}{ $\P( b_x=b_y=1 \g x,y, \cd) = $ } \hspace{-5pt}&
	\begin{blockarray}{rcc}
		\begin{block}{ r |c|c| }
			\cline{2-3}
			{\scriptstyle 1}  & 0 & \sfrac 3 4 \bigstrut \\
			\cline{2-3}
			{\scriptstyle 2} & 1 & 1 \bigstrut\\
			\cline{2-3}
		\end{block}
		\  & {\scriptstyle 1} & {\scriptstyle 2}
	\end{blockarray}
	\qquad
	\raisebox{10pt}{ $\P( b_x=1, b_y=0 \g x,y, \cd) = $ } \hspace{-5pt}&
	\begin{blockarray}{rcc}
		\begin{block}{ r |c|c| }
			\cline{2-3}
			{\scriptstyle 1}  & 1 & \sfrac 1 4\bigstrut \\
			\cline{2-3}
			{\scriptstyle 2} &  0 & 0 \bigstrut \\
			\cline{2-3}
		\end{block}
		\  & {\scriptstyle 1} & {\scriptstyle 2}
	\end{blockarray}
\end{alignat*}
\vspace{-1.5em}
\ec{
	\raisebox{11pt}{ $\P( b_x=0, b_y=1 \g x,y, \cd) = \P( b_x=b_y=0 \g x,y, \cd) = $ } \hspace{0pt}
	\hspace{-5pt}
	\begin{blockarray}{rcc}
		\begin{block}{ r |c|c| }
			\cline{2-3}
			{\scriptstyle 1}  & 0 & 0  \\
			\cline{2-3}
			{\scriptstyle 2} &0& 0 \\
			\cline{2-3}
		\end{block}
		\  & {\scriptstyle 1} & {\scriptstyle 2}
	\end{blockarray}
}
Thus we see that the methods of Section~\ref{sec:kercoup} provide a simple way to
derive a proposal coupling $\bq$ and an acceptance coupling $\bb$ which reproduce the given MH
transition kernel coupling $\bp$.

\subsection{Maximal kernel coupling generated by non-maximal proposal coupling}
\label{ex:nonmax}
We conclude with a slightly more complicated example to demonstrate the properties of maximal
couplings of the MH transition kernel. Assume $\calX = \{1,2,3\}, \scrF = 2^\calX, \sas \xy =
(1,2)$. From these states, we assume the following proposal and transition kernel distributions:
\eq{
	Q(x=1, \cd) = (0,\h,\h) \qquad & P(x=1, \cd) = (\h,\h,0) \\
	Q(y=2,\cd) = (\h,0,\h) \qquad & P(y=2, \cd) = (\h,0,\h).
}
It is easy to confirm that these distributions correspond to an MH transition kernel obtained when
$Q(3, \cd) = (0,1,0)$ and $\pi = (\sfrac{2}{5}, \sfrac{2}{5}, \sfrac{1}{5})$. Simple algebra shows
that any coupling in $\Gmax(Q,Q)$ and any coupling in $\Gmax(P,P)$ must have the following
properties:
\eq{
	\raisebox{3pt}{ $ \bq((1,2), \cd) = $ } \hspace{-5pt}
	\begin{blockarray}{rcccl}
		& {\stackrel{x}{\scriptstyle \downarrow}} & & &  \\
		\begin{block}{ r |c|c|c| l }
			\cline{2-4}
			{\scriptstyle 1} & 0 & \h & 0& \bigstrut \\
			\cline{2-4}
			{\scriptstyle 2} & 0 &0 & 0 & {\hspace{-4pt} \scriptstyle \leftarrow y} \bigstrut  \\
			\cline{2-4}
			{\scriptstyle 3} & 0 & 0 & \h & \bigstrut  \\
			\cline{2-4}
		\end{block}
		\  & {\scriptstyle 1} & {\scriptstyle 2} & {\scriptstyle 3}
	\end{blockarray}
	\hspace{2.5em}
	\raisebox{3pt}{ $ \bp((1,2), \cd) = $ } \hspace{-5pt}	\begin{blockarray}{rcccl}
		& {\stackrel{x}{\scriptstyle \downarrow}} & & &  \\
		\begin{block}{ r |c|c|c| l }
			\cline{2-4}
			{\scriptstyle 1} & \h & 0 & 0 & \bigstrut  \\
			\cline{2-4}
			{\scriptstyle 2} & 0 & 0 & 0 & {\hspace{-4pt} \scriptstyle \leftarrow y}\bigstrut   \\
			\cline{2-4}
			{\scriptstyle 3} & 0 & \h & 0 & \bigstrut  \\
			\cline{2-4}
		\end{block}
		\  & {\scriptstyle 1} & {\scriptstyle 2} & {\scriptstyle 3}
	\end{blockarray}
}
There exists no indicator coupling $\bb$ such that it and the unique maximal $\bq$ given
above generate $\bp$, since the combination of ${\xyp \sim \bq((1,2),\cd), (b_x,b_y) \sim \bb((1,2), \xyp)}$,
$(X, Y) = (b_x x' + (1-b_x) x, b_y y' + (1-b_y) y)$,
and ${(X,Y) \sim \bp((1,2), \cd)}$ yield a contradiction:
\eq{
	\h = \P( X=2,Y=3 \g x=1, y=2)
	\leq \P( x'=2,y'=3 \g x=1, y=2) = 0.
}
Note that in line with Theorem~\ref{thm:repr}, we can still generate $\bp((1,2),\cd)$ using the
following non-maximal proposal coupling and the coupled acceptance indicators $b_x, b_y$:
\eq{
	\raisebox{3pt}{ $ \tilde Q ((1,2), \cd) = $ } \hspace{-5pt}
	\begin{blockarray}{rcccl}
		& {\stackrel{x}{\scriptstyle \downarrow}} & & &  \\
		\begin{block}{ r |c|c|c| l }
			\cline{2-4}
			{\scriptstyle 1} & 0 & 0 & \h & \bigstrut  \\
			\cline{2-4}
			{\scriptstyle 2} & 0 & 0 & 0 & {\hspace{-4pt} \scriptstyle \leftarrow y} \bigstrut  \\
			\cline{2-4}
			{\scriptstyle 3} & 0 & \h & 0 & \bigstrut  \\
			\cline{2-4}
		\end{block}
		\  & {\scriptstyle 1} & {\scriptstyle 2} & {\scriptstyle 3}
	\end{blockarray}
}
\vspace{-2.5em}
\ec{
	\P(b_x =1, b_y = 1 \g (x,y) = (1,2), (x',y') = (2,3)) = 1\\
	\P(b_x = 0, b_y = 1 \g (x,y) = (1,2), (x',y') = (3,1)) = 1.
}
A qualitative conclusion from Lemma~\ref{lem:maxtonon} and the example above is that maximal
couplings $\bp \in \Gmax(P,P)$ may require a certain amount of proposal probability on the diagonal,
but the maximality of $\bq$ is neither necessary nor sufficient for $\bq$ to be able to generate
$\bp$.
\end{appendix}

\begin{acks}[Acknowledgments]
The authors thank Pierre E. Jacob, Persi Diaconis, and Qian Qin for helpful comments.
\end{acks}

\begin{funding}
The first author was supported in part by NSF Grant DMS-1844695. The second author was supported in
part by NSF Grant DMS-2210849.
\end{funding}

%% or include bibliography directly:
\bibliography{refs}{}

\end{document}